\documentclass [a4paper, 12pt, reqno]{amsart}
\usepackage [latin1]{inputenc}
\usepackage {a4}
\usepackage{amscd}
\usepackage{epsfig}
\usepackage{amssymb}
\usepackage{amsmath}
\usepackage{amsthm}
\usepackage[T1]{fontenc}
\usepackage{ae,aecompl}
\usepackage[arrow, matrix, curve]{xy}

\usepackage{color}

\usepackage{geometry}
\newcommand{\C} {\ensuremath{\mathbb{C}}}

\newcommand{\OO}{\mathcal{O}}
\renewcommand{\o}[1]{\overline{#1}}

\newcommand{\dq}{\overline{\partial}}

\DeclareMathOperator{\Sing}{Sing}

\DeclareMathOperator{\Dom}{Dom}

\newtheorem {satz} {Satz} [section]
\newtheorem {lem} [satz] {Lemma}
\newtheorem {cor} [satz] {Corollary}
\newtheorem {defn} [satz] {Definition}
\newtheorem {prop} [satz] {Proposition}

\newtheorem {thm} [satz] {Theorem}

\DeclareMathOperator{\supp}{supp}

\renewcommand{\Im}{\mbox{Im }}

\renewcommand{\theta}{\vartheta}

\title[$L^2$-Serre duality and rational singularities] 
{$L^2$-Serre duality on singular complex spaces and rational singularities}

\author{J. Ruppenthal}

\address{Department of Mathematics, University of Wuppertal, Gau{\ss}str. 20, 42119 Wuppertal, Germany.}
\email{ruppenthal@uni-wuppertal.de}

\date{\today}



\begin{document}

\begin{abstract} 
In the present paper, we devise a version of topological $L^2$-Serre\linebreak
duality for singular complex spaces with arbitrary singularities.
This duality is used to deduce various new $L^2$-vanishing theorems for the $\dq$-equation on singular spaces.
It is shown that complex spaces with rational singularities behave quite tame with respect to the $\dq$-equation
in the $L^2$-sense. More precisely: a singular point is rational if and only if the $L^2$-$\dq_s$-complex is exact in this point.
So, we obtain an $L^2$-$\dq$-resolution of the structure sheaf in rational singular points.
\end{abstract}

\maketitle

~\\[-12mm]

\section{Introduction}

The Cauchy-Riemann operator $\dq$ plays a fundamental role in Complex Analysis and Complex Geometry.
On complex manifolds, functions --  or more generally distributions -- are holomorphic if and only if they are in the kernel of the $\dq$-operator,
and the same holds in a certain sense on normal complex spaces.
For forms of arbitrary degree, the importance of the $\dq$-operator appears strikingly for example in the notion of $\dq$-cohomology
which can be used to represent the cohomology of complex manifolds by the Dolbeault isomorphism.
Solving $\dq$-equations, i.e., $\dq$-vanishing theorems, play a central role in a vast number of problems in Complex Analysis and Geometry,
let us just mention the Cousin problems.

However, a huge part of the $\dq$-theory is still restricted to the smooth setting.
In the present paper, we introduce a version of topological $L^2$-Serre duality for $\dq$-cohomology classes on singular complex spaces with arbitrary singularities,
which allows to deduce some new $L^2$-vanishing theorems for the $\dq$-operator on singular spaces.
We obtain 
an explicit $L^2$-$\dq$-resolution of the structure sheaf of spaces with rational singularities
(and an $L^2$-$\dq$-characterization of rationality).

The $L^2$-theory for the $\dq$-operator is of particular importance in Complex Analysis and Geometry
and has become indispensable for the subject after the fundamental work of
H\"ormander on $L^2$-estimates and existence theorems for the $\dq$-operator \cite{Hoe1}
and the related work of Andreotti and Vesentini \cite{AnVe}.
Important applications of the $L^2$-theory are e.g. the Ohsawa-Takegoshi extension theorem \cite{OT},
Siu's analyticity of the level sets of Lelong numbers \cite{Siu0}
or the invariance of plurigenera \cite{Siu} -- just to name some.

On the other hand, it is almost dispensable to mention the importance of Serre duality at all.
It is one of the most important tools in Complex Algebraic Geometry and Complex Analysis.
On singular spaces, the classical Serre duality has to be replaced by the more involved Grothendieck duality
(developed by Ramis and Ruget in the analytic setting \cite{RR}).
In the present paper, we propose a version of
topological $L^2$-Serre duality which is well-adapted to singular spaces.


\medskip
The first problem one has to face when studying the $\dq$-equation on singular spaces
is that it is not clear what kind of differential forms and operators one should consider.
Recently, there has been some progress by different approaches.

Andersson and Samuelsson developed in \cite {AS} Koppelman integral formulas for the $\dq$-equation on arbitrary singular complex spaces
which allow for a $\dq$-resolution of the structure sheaf in terms of certain fine sheaves of currents, called $\mathcal{A}$-sheaves.
These $\mathcal{A}$-sheaves are defined by an iterative procedure of repeated application of singular integral operators.
A second, somewhat more explicit approach is as follows: consider
differential forms which are defined on the regular part of a singular variety
and which are square-integrable up to the singular set.
This setting seems to be quite fruitful and has some history by now (see \cite{PS1}).
Also in this direction, some progress has been made recently.
{\O}vrelid--Vassiliadou and the author obtained in \cite{OV3} and \cite{eins} a pretty complete description of the $L^2$-cohomology
of the $\dq$-operator (in the sense of distributions) at isolated singularities.

In this setting, we understand the class of objects with which we deal very well (just $L^2$-forms), 
but the disadvantage is a different one. Whereas the $\dq$-equation is locally solvable for $\dq$-closed $(0,q)$-forms
in the category of $\mathcal{A}$-sheaves by the Koppelman formulas in \cite{AS}, 
there are local obstructions to solving the $\dq$-equation in the $L^2$-sense at singular points (see e.g. \cite{FOV2},  \cite{OV3}, \cite{eins}).
So, there can be no $L^2$-$\dq$-resolution for the structure sheaf in general.

The starting point of the present paper was the idea that the $\dq$-operator in the $L^2$-sense
may behave very well on spaces with canonical singularities. There are several results which substantiate this idea.
First, reflexive differential forms are weakly holomorphic on such spaces,
i.e., they can be represented as the push-forward of holomorphic forms under a resolution of singularities (\cite{GKKP}, Theorem 1.4).
It follows that holomorphic differential forms on the regular part of the variety are square integrable.
Second, it is known that the $\dq$-equation is solvable in the $L^2$-sense for $(0,q)$-forms
on affine cones over projective varieties of low degree\footnote{Such varieties are canonical.}
(see \cite{Rp7}, \cite{OV3}, Section 7, Example 3, \cite{LR} and \cite{LR2}).
Third, there is a long list of related positive results on the $\dq$-equation at rational double points\footnote{
Rational double points are precisely the possible canonical singularities in dimension two.}
(see \cite{AZ1}, \cite{AZ2}, \cite{AS},
and also \cite{Rp7}, \cite{LR} for the $A_1$-singularity).

The underlying idea is that canonical singularities are rational (see e.g. \cite{Kol}, Theorem 11.1),
i.e., we expect that the singularities do not contribute to the local cohomology in a certain sense.
Pursuing this idea, it turned out that there is a notion of $L^2$-$\dq$-cohomology for $(0,q)$-forms
which can be described completely in terms of a resolution of singularities (see Theorem \ref{thm:main3a} below).
A singular point is rational if and only if this very $L^2$-$\dq$-complex is exact in that point (Corollary \ref{cor:main2}).
If the underlying space has rational singularities, 
then we obtain an $L^2$-$\dq$-resolution of the structure sheaf.
This gives even more hope for the idea that spaces with canonical singularities may allow for
a comprehensive $L^2$-theory for the $\dq$-operator.

One of our main tools is a version of topological $L^2$-Serre duality for singular complex spaces with arbitrary singularities.
In order to establish and apply this kind of Serre duality, we have to overcome several difficulties
which do not appear in the smooth setting.
We have to establish some new duality relations between different closed $L^2$-extensions of the $\dq$-operator
at arbitrary singularities, and to provide new Hausdorffness results for some cohomology groups
on singular spaces.

\medskip
To explain our results more precisely,
let us introduce some notation (see Section~\ref{ssec:complexes} for the details).
Let $X$ be a Hermitian complex space\footnote{A Hermitian complex space $(X,g)$ is a reduced complex space $X$ with a metric $g$ on the regular part
such that the following holds: If $x\in X$ is an arbitrary point there exists a neighborhood $U=U(x)$ and a
biholomorphic embedding of $U$ into a domain $G$ in $\C^N$ and an ordinary smooth Hermitian metric in $G$
whose restriction to $U$ is $g|_U$.}
of pure dimension $n$ and $F\rightarrow X$ a Hermitian holomorphic line bundle.
We denote by $\mathcal{L}^{p,q}(F)$ the sheaf of germs of $F$-valued $(p,q)$-forms on the regular part of $X$
which are square-integrable on $K\setminus\Sing X$ for compact sets $K$.\footnote{This is what we mean by
square-integrable up to the singular set.}

Due to the incompleteness of the metric on $X\setminus \Sing X$, there are different reasonable definitions
of the $\dq$-operator on $\mathcal{L}^{p,q}(F)$-forms. To be more precise, let $\dq_{cpt}$ be the $\dq$-operator
on smooth forms with support away from the singular set $\Sing X$.
Then $\dq_{cpt}$ can be considered as a densely defined operator $\mathcal{L}^{p,q}(F) \rightarrow \mathcal{L}^{p,q+1}(F)$.
One can now consider various closed extensions of this operator. The two most important are the maximal closed extension,
i.e., the $\dq$-operator in the sense of distributions which we denote by $\dq_w$,
and the minimal closed extension, i.e., the closure of the graph of $\dq_{cpt}$ which we denote by $\dq_s$.
Let $\mathcal{C}^{p,q}(F)$ be the domain of definition of $\dq_w$ which is a subsheaf of $\mathcal{L}^{p,q}(F)$,
and $\mathcal{F}^{p,q}(F)$ the domain of definition of $\dq_s$ which in turn is a subsheaf of $\mathcal{C}^{p,q}(F)$.
We obtain complexes of fine sheaves
\begin{eqnarray}\label{eq:intro1}
\mathcal{C}^{p,0}(F) \overset{\dq_w}{\longrightarrow} \mathcal{C}^{p,1}(F) \overset{\dq_w}{\longrightarrow} 
\mathcal{C}^{p,2}(F) \overset{\dq_w}{\longrightarrow} ...
\end{eqnarray}
and
\begin{eqnarray}\label{eq:intro2}
\mathcal{F}^{p,0}(F) \overset{\dq_s}{\longrightarrow} \mathcal{F}^{p,1}(F) \overset{\dq_s}{\longrightarrow} \mathcal{F}^{p,2}(F) \overset{\dq_s}{\longrightarrow} ...
\end{eqnarray}

If $F$ is just the trivial line bundle, then $\mathcal{K}_X:=\ker\dq_w \subset \mathcal{C}^{n,0}$ is the canonical sheaf of Grauert--Riemenschneider
and $\mathcal{K}_X^s:=\ker\dq_s \subset \mathcal{F}^{n,0}$ is the sheaf of holomorphic $n$-forms with a certain boundary condition
that was introduced in \cite{eins}.
We will see below that $\widehat{\OO}_X=\ker\dq_s \subset \mathcal{F}^{0,0}$ for the sheaf of weakly holomorphic functions $\widehat{\OO}_X$.

It is clear that \eqref{eq:intro1} and \eqref{eq:intro2} are exact in regular points of $X$.
Exactness in singular points is equivalent to the difficult problem of solving $\dq$-equations locally in the $L^2$-sense at singularities,
which is not possible in general (see e.g. \cite{FOV2}, \cite{OV1}, \cite{OV3}, \cite{Rp1}, \cite{Rp7}, \cite{eins}).
However, it is known that \eqref{eq:intro1} is exact for $p=n$,
and that \eqref{eq:intro2} is exact for $p=n$ if $X$ has only isolated singularities
(\cite{PS1} and \cite{eins}; see Section \ref{ssec:solvability}).
In these cases, the complexes \eqref{eq:intro1} and \eqref{eq:intro2} are fine resolutions of the canonical sheaves $\mathcal{K}_X$ and $\mathcal{K}_X^s$,
respectively.

For an open set $\Omega\subset X$, we denote by $H^{p,q}_{w,loc}(\Omega,F)$ the cohomology of the complex \eqref{eq:intro1},
and by $H^{p,q}_{w,cpt}(\Omega,F)$ the cohomology of \eqref{eq:intro1} with compact support. Analogously, let $H^{p,q}_{s,loc}(\Omega,F)$
and $H^{p,q}_{s,cpt}(\Omega,F)$ be the cohomology groups of \eqref{eq:intro2}.
These $L^2$-cohomology groups inherit the structure of topological vector spaces,
which are locally convex Hausdorff spaces if the corresponding $\dq$-operators have closed range.\footnote{
Note that different Hermitian metrics lead to $\dq$-complexes which are equivalent on relatively compact subsets.
So, one can put any Hermitian metric on $X$ in many of the results below.}

We can now formulate the main results of the present paper.

\begin{thm}\label{thm:main3a}
Let $X$ be a Hermitian complex space
, $\pi: M\rightarrow X$ a resolution of singularities
and $\Omega\subset X$ a holomorphically convex domain. Then push-forward of forms induces for all $q\geq0$ a natural topological isomorphism
\begin{eqnarray}\label{eq:main3a}
H^q\big(\pi^{-1}(\Omega),\OO_M\big) 
\overset{\cong}{\longrightarrow} H^{0,q}_{s,loc}(\Omega).
\end{eqnarray}
\end{thm}

From that we obtain for the local $L^2$-$\dq_s$-cohomology immediately:

\begin{cor}\label{cor:main1}
Let $X$ be a Hermitian complex space, $\pi: M \rightarrow X$ a resolution of singularities and $q\geq0$.
Then:
\begin{eqnarray}\label{eq:cormain1}
\big(\mathcal{H}^q (\mathcal{F}^{0,*})\big)_x \cong \big(R^q \pi_* \OO_M\big)_x\ \ \ \ \ \forall x\in X.
\end{eqnarray}
Particularly, for $q=0$,
\begin{eqnarray}\label{eq:weaklyholomorphic}
\ker \dq_s^{0,0} = \mathcal{H}^0(\mathcal{F}^{0,*}) = \widehat{\OO}_X,
\end{eqnarray}
where $\widehat{\OO}_X$ denotes the sheaf of germs of weakly holomorphic functions.
\end{cor}

It follows that $x\in X$ is a normal point exactly if $\big(\ker \dq_s^{0,0}\big)_x= \OO_{X,x}$.
As $x\in X$ is a rational point if it is normal and $(R^q\pi_* \OO_M)_x=0$ for $q>0$, we get also:

\begin{cor}\label{cor:main2}
Let $X$ be a Hermitian complex space. Then the $L^2$-$\dq$-complex
\begin{eqnarray}\label{eq:exactnessM01}
0\rightarrow \OO_X \longrightarrow \mathcal{F}^{0,0} \overset{\dq_s}{\longrightarrow}
\mathcal{F}^{0,1} \overset{\dq_s}{\longrightarrow} \mathcal{F}^{0,2} 
\overset{\dq_s}{\longrightarrow} \mathcal{F}^{0,3} \overset{\dq_s}{\longrightarrow} ...
\end{eqnarray}
is exact in a point $x\in X$ if and only if $x$ is a rational point.

Hence, if $X$ has only rational singularities, then \eqref{eq:exactnessM01}
is a fine resolution of the structure sheaf $\OO_X$. 
\end{cor}

If $X$ has only rational singularities, then Corollary \ref{cor:main2} yields directly further finiteness and vanishing results,
e.g. if $X$ is $q$-convex or $q$-complete. 
An essential tool in the proof of Theorem \ref{thm:main3a} is the following new version of topological $L^2$-Serre duality
on singular complex spaces with arbitrary singularities:

\begin{thm}[{\bf $L^2$-Serre duality}]\label{thm:main1}
Let $X$ be a Hermitian complex space of pure dimension $n$, $\Omega\subset X$ an open set,
$F\rightarrow X$ a Hermitian holomorphic line bundle, and let $0\leq p,q \leq n$.
If $H^{p,q}_{w,loc}(\Omega,F)$ and $H^{p,q+1}_{w,loc}(\Omega,F)$ are Hausdorff, then the mapping
\begin{eqnarray}\label{eq:serreduality}
\mathcal{L}^{p,q}(\Omega,F) \times \mathcal{L}^{n-p,n-q}_{cpt}(\Omega,F^*) \rightarrow \C\ \ ,\ (\eta,\omega) \mapsto \int_{\Omega^*} \eta\wedge\omega,
\end{eqnarray}
induces a non-degenerate pairing
\begin{eqnarray*}
H^{p,q}_{w,loc}(\Omega,F) \times H^{n-p,n-q}_{s,cpt}(\Omega,F^*) \rightarrow \C
\end{eqnarray*}
under which
$H^{n-p,n-q}_{s,cpt}(\Omega,F^*)$ is algebraically isomorphic to the dual space of $H^{p,q}_{w,loc}(\Omega,F)$ and vice versa.

If $H^{p,q}_{s,loc}(\Omega,F)$ and $H^{p,q+1}_{s,loc}(\Omega,F)$ are Hausdorff, then \eqref{eq:serreduality}
induces a non-degenerate pairing
\begin{eqnarray*}
H^{p,q}_{s,loc}(\Omega,F) \times H^{n-p,n-q}_{w,cpt}(\Omega,F^*) \rightarrow \C
\end{eqnarray*}
under which $H^{n-p,n-q}_{w,cpt}(\Omega,F^*)$ is algebraically isomorphic to the dual space of $H^{p,q}_{s,loc}(\Omega,F)$ and vice versa.
\end{thm}

\smallskip
If the topological vector spaces $H^{p,q}_{w/s,loc}(\Omega,F)$, $H^{p,q+1}_{w/s,loc}(\Omega,F)$ are non-Hausdorff,
then the statement of Theorem \ref{thm:main1} holds at least for the separated cohomology groups $\overline{H}_{w/s} = \ker \dq_{w/s} / \overline{\Im \dq_{w/s}}$
(Theorem \ref{thm:duality1} and Theorem \ref{thm:duality2}).\footnote{The notation $w/s$ refers either to the index $w$ or the index $s$ in the whole statement.}
Though Theorem \ref{thm:main1} looks at first glance like a standard version of Serre duality,
there are two essential difficulties in its proof which do usually not appear.
First, the $\dq$-operators under consideration are just
closed densely defined operators in the Fr\'echet spaces $\mathcal{L}^{p,q}(\Omega,F)$ and the $(LF)$-spaces $\mathcal{L}^{n-p,n-q}_{cpt}(\Omega,F^*)$.\footnote{Usually, one considers either closed densely defined operators between Banach spaces,
or operators between Fr\'echet-Schwartz, i.e. (FS), and $(DFS)$-spaces which are defined on the whole spaces, not just on dense subsets.
This setting is considered e.g. in \cite{Serre}, \cite{RR} or \cite{AK}. (FS) and (DFS)-spaces have nicer properties
than Fr{\'e}chet and (LF)-spaces.}
Second, we have to show that the operators $\dq_w$ and $\dq_s$ are topologically dual at arbitrary singularities.
This point is the main technical difficulty in this paper.

Note that $H^{p,q}_{w/s,loc}(\Omega,F)$ is Hausdorff if and only if $\dq_{w/s}$ has closed range in $\mathcal{L}^{p,q}(\Omega,F)$,
and to decide whether this is the case is usually as difficult as solving the corresponding $\dq$-equation.
Using local $L^2$-$\dq$-solution results for singular spaces and the theory of Fr\'echet sheaves, we will show at least:

\begin{thm}\label{thm:main2}
Let $X$ be a Hermitian complex space of pure dimension $n$, $F\rightarrow X$ a Hermitian holomorphic line bundle, and let $0\leq p,q \leq n$.
Let $\Omega \subset X$ be a holomorphically convex open subset. Then the topological vector spaces
$$H^{n,q}_{w,loc}(\Omega,F)\ \ ,\ \ H^{n,q}_{w,cpt}(\Omega,F)\ \ ,\ \ H^{0,n-q}_{s,cpt}(\Omega,F^*)\ \ ,\ \ H^{0,n-q}_{s,loc}(\Omega,F^*)$$
are Hausdorff for all $0\leq q\leq n$.

If $X$ has only isolated singularities, then the topological vector spaces
$$H^{n,q}_{s,loc}(\Omega,F)\ \ ,\ \ H^{0,n-q}_{w,cpt}(\Omega,F^*)$$
are Hausdorff for all $0\leq q\leq n$, too.

If $X$ has only homogeneous (conical) isolated singularities, then the topological vector spaces
$$H^{n,q}_{s,cpt}(\Omega,F)\ \ ,\ \ H^{0,n-q}_{w,loc}(\Omega,F^*)$$
are Hausdorff for all $0\leq q\leq n$, too.
\end{thm}

A main point in the proof of Theorem \ref{thm:main2} is to show that the canonical Fr\'echet sheaf structure
of compact convergence on the coherent analytic canonical sheaves $\mathcal{K}_X$ and $\mathcal{K}_X^s$, respectively,
coincides with the Fr\'echet sheaf structure of $L^2$-convergence on compact subsets (Theorem \ref{thm:topologies1}).
This allows then to show also the topological equivalence of \v{C}ech cohomology and $L^2$-cohomology (Theorem \ref{thm:topology}).
Another considerable, non-standard difficulty in the proof of Theorem {\ref{thm:main2} is to show that $H^{n-p,n-q+1}_{s/w,loc}(\Omega,F^*)$
is Hausdorff if $H^{p,q}_{w/s,cpt}(\Omega,F)$ is Hausdorff (see Lemma \ref{lem:range2}).

As a direct application of our $L^2$-Serre duality, Theorem \ref{thm:main1},
we deduce again by use of local $\dq$-solution results for $(n,q)$-forms and the equivalence of Fr\'echet structures, Theorem \ref{thm:topology},
another main result:

\begin{thm}\label{thm:vanishing1}
Let $X$ be a Hermitian complex space of pure dimension $n$, $F\rightarrow X$ a Hermitian holomorphic line bundle
and $\Omega\subset X$ a cohomologically $q$-complete open subset, $q\geq 1$. Then
\begin{eqnarray*}
H^{n,r}_{w,loc}(\Omega,F) = H^{0,n-r}_{s,cpt}(\Omega,F^*)=0\ \ \ \mbox{ for all }\  r\geq q.
\end{eqnarray*}
If $X$ has only isolated singularities, then also
\begin{eqnarray*}
H^{n,r}_{s,loc}(\Omega,F) = H^{0,n-r}_{w,cpt}(\Omega,F^*)=0\ \ \ \mbox{ for all }\  r\geq q.
\end{eqnarray*}
\end{thm}

Note that $\Omega$ is cohomologically $q$-complete if it is $q$-complete by the Andreotti-Grauert vanishing theorem \cite{AG}.
So, Theorem \ref{thm:vanishing1} allows to solve the $\dq_s$-equation with compact support for $(0,n-q)$-forms on $q$-complete spaces,
which is of particular interest for $1$-complete spaces, i.e., Stein spaces.

It is thus interesting to understand the $\dq_s$-operator better.
In contrast to the $\dq$-operator in the sense of distributions (the $\dq_w$-operator),
$\dq_s$ comes with a certain kind of boundary (respectively growth) condition at the singular set.
It follows by an argument of N. Sibony from \cite{Sibony}
that locally bounded forms 
in the domain of the $\dq_w$-operator are also in the domain of the $\dq_s$-operator (Theorem \ref{thm:domaindqs2}).
This is a direct way to see that weakly holomorphic functions are $\dq_s$-closed (cf. Corollary \ref{cor:main1}).

\smallskip
To exemplify the use of $\dq$-equations on singular spaces,
we give as an application a short proof of the Hartogs' extension theorem in its most general form:

\begin{thm}\label{thm:hartogs}
Let $X$ be a connected normal complex space of dimension $n\geq 2$ which is cohomologically $(n-1)$-complete.
Furthermore, let $D$ be a domain in $X$ and $K\subset D$ a compact subset such that $D\setminus K$ is connected.
Then each holomorphic function $f\in \mathcal{O}(D\setminus K)$ has a unique holomorphic extension to the whole set $D$.
\end{thm}

In almost this generality, i.e., for $(n-1)$-complete spaces,
Hartogs' extension theorem is due to Merker and Porten \cite{MePo2},
who proved also extension of meromorphic functions.
Merker and Porten gave an involved geometrical proof by using a finite number of
parameterized families of holomorphic discs and Morse-theoretical tools for the global topological control 
of monodromy, but no $\dq$-theory.
Shortly after that, Col\c{t}oiu and Ruppenthal obtained a $\dq$-theoretical proof of the slightly more general Theorem \ref{thm:hartogs}
by using the Ehrenpreis-$\dq$-technique
on a resolution of singularities (see \cite{CR}).
It is a natural question to ask whether the extension can be achieved by working on the original singular space only.
In the present paper, we obtain a very short proof of Theorem \ref{thm:hartogs} by the Ehrenpreis-$\dq$-technique.
We just use the $\dq$-vanishing $H^{0,1}_{s,cpt}(X)=0$
for an $(n-1)$-complete space $X$ and the fact that bounded $\dq$-closed forms are in the kernel of $\dq_s$ (see Section \ref{ssec:hartogs}).
A similar proof has been given also by {\O}vrelid and Vassiliadou (\cite{OV4}, Theorem 1.3)
by the use of their weighted $L^2$-solvability results for the $\dq$-equation on singular spaces \cite{OV4}.

We also show that $H^{0,1}_{L^\infty,cpt}(X)=0$ for an $(n-1)$-complete space $X$ 
by using a resolution of singularities and Takegoshi's vanishing theorem (Theorem \ref{thm:vanishing3}).

\smallskip
Let us point out also the following interesting fact.
Let $X$ be a Gorenstein space with canonical singularities.
By exactness of \eqref{eq:exactnessM01} and exactness of \eqref{eq:intro1} for $p=n$,
the non-degenerate $L^2$-Serre duality pairing
$$H^{0,q}_{s,loc}(\Omega) \times H^{n,n-q}_{w,cpt}(\Omega) \rightarrow \C\ ,\ ([\eta],[\omega]) \mapsto \int_{\Omega^*} \eta\wedge \omega,$$
is for $0\leq q\leq n$ then an explicit realization of Grothendieck duality after Ramis-Ruget \cite{RR},
$$H^q(\Omega,\OO_X) \cong H^{n-q}_{cpt}(\Omega,\omega_X),$$
given the cohomology groups under consideration are Hausdorff.
Here, $\omega_X$ denotes the Grothendieck dualizing sheaf which coincides with the Grauert-Riemenschneider canonical sheaf $\mathcal{K}_X$
as $X$ has canonical Gorenstein singularities.

\medskip
The present paper is organized as follows. In Section \ref{sec:preliminaries} we provide the necessary preliminaries:
the $\dq_w$- and the $\dq_s$-complex, some $L^2$-$\dq$-results and consequences, 
$L^2$-$\dq$-Hilbert space theory, topological preliminaries,
Fr\'echet sheaves. In Section \ref{sec:serre}, we prove Serre duality, Theorem \ref{thm:main1}, 
and study the equivalence of the topology of compact convergence and $L^2$-topology which leads to Theorem \ref{thm:main2}
and Theorem \ref{thm:vanishing1}. Section \ref{sec:dqs} is then devoted to the study of the $\dq_s$-operator and 
Hartogs' extension theorem.
Finally, we prove Theorem \ref{thm:main3a} and its corollaries
in the last section.

\newpage
\section{Preliminaries}\label{sec:preliminaries}

\subsection{Two $\dq$-complexes on singular spaces}\label{ssec:complexes}

Let us recall some of the essential constructions from \cite{eins}.

Let $X$ be a (singular) Hermitian complex space of pure dimension $n$.
For any subset $S\subset X$, we use the notation $S^*$ for $S\setminus \Sing X$.
Let $F\rightarrow X^*$ be a Hermitian holomorphic line bundle and $U\subset X$ an open subset.
On a singular space, it is fruitful
to consider forms that are square-integrable up to the singular set.
Hence, we use the following concept of locally square-integrable forms:
\begin{eqnarray*}
L_{loc}^{p,q}(U,F):=\{f \in L_{loc}^{p,q}(U^*,F): f|_K \in L^{p,q}(K^*,F)\ \forall K\subset\subset U\}.
\end{eqnarray*}
It is easy to check that the presheaves given as
$$\mathcal{L}^{p,q}(U,F) := L_{loc}^{p,q}(U,F)$$
are already sheaves $\mathcal{L}^{p,q}(F)\rightarrow X$. On $L_{loc}^{p,q}(U,F)$, we denote by
$$\dq_w(U): L_{loc}^{p,q}(U,F) \rightarrow L_{loc}^{p,q+1}(U,F)$$
the $\dq$-operator in the sense of distributions on $U^*=U\setminus\Sing X$ which is closed and densely defined.
When there is no danger of confusion, we will simply write $\dq_w$ for $\dq_w(U)$.
The subscript refers to $\dq_w$ as an operator in a weak sense.
Since $\dq_w$ is a local operator, i.e. $\dq_w(U)|_V = \dq_w(V)$
for open sets $V\subset U$,
we can define the presheaves of germs of forms in the domain of $\dq_w$,
$$\mathcal{C}^{p,q}(F):=\mathcal{L}^{p,q}(F)\cap \dq_w^{-1}\mathcal{L}^{p,q+1}(F),$$
given by
$$\mathcal{C}^{p,q}(U,F) = \mathcal{L}^{p,q}(U,F) \cap\Dom\dq_w(U).$$
These are actually already sheaves
because the following is also clear: If $U=\bigcup U_\mu$ is a union of open sets, $f_\mu=f|_{U_\mu}$ and
$$f_\mu \in \Dom \dq_w(U_\mu),$$
then
$$f\in \Dom \dq_w(U)\ \ \  \mbox{ and }\ \ \  \big(\dq_w(U) f\big)|_{U_\mu} = \dq_w(U_\mu) f_\mu.$$
Moreover, it is easy to see that the sheaves $\mathcal{C}^{p,q}(F)$ admit partitions of unity,
and so we obtain a complex of fine sheaves
\begin{eqnarray}\label{eq:Cseq1}
\mathcal{C}^{p,0}(F) \overset{\dq_w}{\longrightarrow} \mathcal{C}^{p,1}(F) \overset{\dq_w}{\longrightarrow} 
\mathcal{C}^{p,2}(F) \overset{\dq_w}{\longrightarrow} ...
\end{eqnarray}
We use simply $\mathcal{C}^{p,q}$ to denote the sheaves of forms with values in the trivial line bundle.
We define
\begin{eqnarray}\label{defn:KX}
\mathcal{K}_X(F) := \ker \dq_w \subset \mathcal{C}^{n,0}(F).
\end{eqnarray}
Using a resolution of singularities, one sees that $\mathcal{K}_X :=\ker\dq_w\subset \mathcal{C}^{n,0}$
is just the canonical sheaf of Grauert and Riemenschneider because the $L^2$-property of $(n,0)$-forms
remains invariant under modifications (see \cite{eins}, Section 2.2).

The $L^{2,loc}$-Dolbeault cohomology for forms with values in $F$ with respect to the $\dq_w$-operator 
on an open set $U\subset X$
is by definition the cohomology of the complex \eqref{eq:Cseq1} which is denoted by $H^q(\Gamma(U,\mathcal{C}^{p,*}(F)))$.
The cohomology with compact support is $H^q(\Gamma_{cpt}(U,\mathcal{C}^{p,*}(F)))$. Note that this is the cohomology
of forms with compact support in $U$, not with compact support in $U^*=U\setminus\Sing X$.

\medskip
We use also the following notation for the $\dq_w$-cohomology:

\begin{defn}\label{defn:Hw}
For an open set $\Omega\subset X$ and a Hermitian holomorphic line bundle $F\rightarrow X$, let
\begin{eqnarray*}
H^{p,q}_{w,loc}(\Omega,F) &:=& H^q\big(\Gamma(\Omega,\mathcal{C}^{p,*}(F))\big),\\
H^{p,q}_{w,cpt}(\Omega,F) &:=& H^q\big(\Gamma_{cpt}(\Omega,\mathcal{C}^{p,*}(F))\big).
\end{eqnarray*}
\end{defn}

Secondly, we introduce a suitable local realization of a minimal version of the $\dq$-operator.
This is the $\dq$-operator with a certain boundary condition at the singular set $\Sing X$ of $X$.
Let
$$\dq_s(U): L_{loc}^{p,q}(U,F) \rightarrow L_{loc}^{p,q+1}(U,F)$$
be defined as follows.\footnote{Again, we write simply $\dq_s$ for $\dq_s(U)$ 
if there is no danger of confusion.} We say that $f\in\Dom\dq_w$ is in the domain of $\dq_s$ if there exists a
sequence of forms $\{f_j\}_j \subset \Dom\dq_w \subset L_{loc}^{p,q}(U,F)$ with essential support away from the singular set,
$$\supp f_j \cap \Sing X = \emptyset,$$
such that
\begin{eqnarray}\label{eq:ds1}
f_j \rightarrow f &\mbox{ in }& L^{p,q}(K^*,F),\\
\dq_w f_j \rightarrow \dq_w f &\mbox{ in }& L^{p,q+1}(K^*,F)\label{eq:ds2}
\end{eqnarray}
for each compact subset $K\subset\subset U$. The subscript refers to $\dq_s$ as an extension in a strong sense.
Note that we can assume without loss of generality
(by use of cut-off functions and smoothing with Dirac sequences) that the forms $f_j$ are smooth with compact support in $U^*=U-\Sing X$.
This is the equivalent definition that we used in \cite{eins}
where we denoted the operator by $\dq_{s,loc}$.

It is now clear that $\dq_s(U)|_V = \dq_s(V)$
for open sets $V\subset U$, and we can define the presheaves of germs of forms in the domain of $\dq_s$,
$$\mathcal{F}^{p,q}(F):=\mathcal{L}^{p,q}(F)\cap \dq_s^{-1}\mathcal{L}^{p,q+1}(F),$$
given by 
$$\mathcal{F}^{p,q}(U,F) = \mathcal{L}^{p,q}(U,F) \cap \Dom \dq_s(U).$$
Here, we shall check a bit more carefully that these are already sheaves:
Let $U=\bigcup U_\mu$ 
be a union of open sets, $f\in L_{loc}^{p,q}(U,F)$ and $f_\mu=f|_{U_\mu} \in \Dom\dq_s(U_\mu)$
for all $\mu$. We claim that $f\in\Dom\dq_s(U)$.
To see this, we can assume (by taking a refinement if necessary)
that the open cover $\mathcal{U}:=\{U_\mu\}_\mu$ is locally finite,
and choose a partition of unity $\{\varphi_\mu\}_\mu$ for $\mathcal{U}$.
On $U_\mu$ choose a sequence $\{f^\mu_j\}_j\subset L_{loc}^{p,q}(U_\mu,F)$ as in \eqref{eq:ds1}, \eqref{eq:ds2},
and consider
$f_j := \sum_{\mu} \varphi_\mu f_j^\mu$.
It is clear that $\{f_j\}_j\subset L_{loc}^{p,q}(U,F)$. If $K\subset\subset U$ is compact, then $K\cap \supp \varphi_\mu$
is a compact subset of $U_\mu$ for each $\mu$, so that $\{f_j^\mu\}_j$ and $\{\dq f_j^\mu\}_j$ converge in the $L^2$-sense
to $f_\mu$ resp. $\dq_w f_\mu$ on $K\cap\supp \varphi_\mu$. But then $\{f_j\}_j$ and $\{\dq f_j\}_j$ converge in the $L^2$-sense
to $f$ resp. $\dq_w f$ on $K$ (recall that the cover is locally finite) and that is what we had to show.

As for $\mathcal{C}^{p,q}(F)$, it is clear that the sheaves $\mathcal{F}^{p,q}(F)$ are fine,
and we obtain a complex of fine sheaves
\begin{eqnarray}\label{eq:Fseq1}
\mathcal{F}^{p,0}(F) \overset{\dq_s}{\longrightarrow} \mathcal{F}^{p,1}(F) \overset{\dq_s}{\longrightarrow} \mathcal{F}^{p,2}(F) \overset{\dq_s}{\longrightarrow} ...
\end{eqnarray}
Again, we use simply $\mathcal{F}^{p,q}$ to denote the sheaves of forms with values in the trivial line bundle.
We define the canonical sheaf of holomorphic $n$-forms with a Dirichlet boundary condition:
\begin{eqnarray*}
\mathcal{K}_X^s (F):= \ker \dq_s \subset \mathcal{F}^{n,0}(F).
\end{eqnarray*}

\begin{defn}\label{defn:Hs}
For $\Omega\subset X$ open, we use the notation:
\begin{eqnarray*}
H^{p,q}_{s,loc}(\Omega,F) &:=& H^q\big(\Gamma(\Omega,\mathcal{F}^{p,*}(F))\big),\\
H^{p,q}_{s,cpt}(\Omega,F) &:=& H^q\big(\Gamma_{cpt}(\Omega,\mathcal{F}^{p,*}(F))\big).
\end{eqnarray*}
\end{defn}

\smallskip
\subsection{Local $L^2$-solvability for $(n,q)$-forms}\label{ssec:solvability}

It is clearly interesting to study whether the sequences \eqref{eq:Cseq1} and \eqref{eq:Fseq1} are exact,
which is well-known to be the case in regular points of $X$ where the $\dq_w$- and the $\dq_s$-operator coincide.
In singular points, the situation is quite complicated for forms of arbitrary degree and not completely understood.
However, the $\dq_w$-equation is locally solvable in the $L^2$-sense at arbitrary singularities for forms
of degree $(n,q)$, $q>0$ (\cite{PS1}, Proposition 2.1), and for forms of degree $(p,q)$, $p+q>n$, 
at isolated singularities (\cite{FOV2}, Theorem 1.2).
We may restrict ourselves here to the case of $(n,q)$-forms and have (see \cite{eins}, Theorem 3.1):

\begin{thm}\label{thm:exactness1}
Let $X$ be a Hermitian complex space of pure dimension $n$,
and $F\rightarrow X^*=X\setminus \Sing X$ a Hermitian holomorphic line bundle which is locally semi-positive with respect to $X$,
i.e. for each point $x\in X$ there is a neighborhood $U_x\subset X$ such that $F$ is semi-positive on $U_x^*=U_x\setminus \Sing X$.
Then the complex
\begin{eqnarray}\label{eq:complex1}
0\rightarrow \mathcal{K}_X(F) \longrightarrow \mathcal{C}^{n,0}(F) \overset{\dq_w}{\longrightarrow}
\mathcal{C}^{n,1}(F) \overset{\dq_w}{\longrightarrow} \mathcal{C}^{n,2}(F) \overset{\dq_w}{\longrightarrow} ...
\end{eqnarray}
is exact, i.e. it is a fine resolution of $\mathcal{K}_X(F)$. For an open set $U\subset X$, it follows that
$$H^q(U,\mathcal{K}_X(F)) \cong H^{n,q}_{w,loc}(U,F)\ \ ,\ \ H^q_{cpt}(U,\mathcal{K}_X(F)) \cong H^{n,q}_{w,cpt}(U,F).$$
\end{thm}

Note that the positivity assumption on $F$ is trivially fulfilled if $F$ extends to a holomorphic line bundle over $X$.
For the case of the trivial line bundle, $F=X\times\C$, Theorem \ref{thm:exactness1} is due to Pardon-Stern (\cite{PS1}, Proposition 2.1).

Concerning the $\dq_s$-equation, local $L^2$-solvability for forms of degree $(n,q)$
is known to hold on spaces with isolated singularities (see \cite{eins}, Lemma 5.4 and Lemma 6.3),
but the problem is open at arbitrary singularities.

So, let $X$ have only isolated singularities. 
Then the $\dq_s$-equation is locally exact on $(n,q)$-forms
for $1\leq q \leq n-1$ by \cite{eins}, Lemma 5.4, and for $q\geq 2$ by \cite{eins}, Lemma 6.3.
Both statements were deduced from the results of Forn{\ae}ss, {\O}vrelid and Vassiliadou \cite{FOV2}.
The case of $\dim X=n=1$ is treated in \cite{RS}. Hence:

\begin{thm}\label{thm:exactness2}
Let $X$ be a Hermitian complex space of pure dimension $n$ with only isolated singularities. Then
\begin{eqnarray}\label{eq:exactness2}
0\rightarrow \mathcal{K}_X^s \hookrightarrow \mathcal{F}^{n,0} \overset{\dq_s}{\longrightarrow}
\mathcal{F}^{n,1} \overset{\dq_s}{\longrightarrow} \mathcal{F}^{n,2} \overset{\dq_s}{\longrightarrow} ... \longrightarrow \mathcal{F}^{n,n}
\rightarrow 0
\end{eqnarray}
is a fine resolution.
For an open set $U\subset X$, it follows that
$$H^q(U,\mathcal{K}_X^s) \cong H^{n,q}_{s,loc}(U)\ \ ,\ \ H^q_{cpt}(U,\mathcal{K}_X^s) \cong H^{n,q}_{s,cpt}(U).$$
\end{thm}

\medskip
\subsection{Resolution of singularities and Takegoshi's vanishing theorem}\label{ssec:resolution}

We need to recall some more material from \cite{eins}, Section 2.2.
For $X$ as above and $F\rightarrow X$ a Hermitian holomorphic line bundle,
let $\pi: M \rightarrow X$ be a resolution of singularities (which exists due to Hironaka),
and give $M$ an arbitrary positive definite Hermitian metric $\sigma$.
Then we denote analogously to \eqref{eq:Cseq1} by
\begin{eqnarray}\label{eq:new1}
0\rightarrow \mathcal{K}_M(\pi^*F) \longrightarrow \mathcal{C}^{n,0}_\sigma(\pi^* F) \overset{\dq_w}{\longrightarrow}
\mathcal{C}^{n,1}_\sigma(\pi^* F) \overset{\dq_w}{\longrightarrow} \mathcal{C}^{n,2}_\sigma(\pi^* F) \overset{\dq_w}{\longrightarrow} ...
\end{eqnarray}
the fine $L^{2,loc}_\sigma$-resolution of the canonical sheaf $\mathcal{K}_M(\pi^* F)$ on $M$ 
with values in $\pi^* F$ (it is well known that \eqref{eq:new1} is exact).
As $\mathcal{K}_X$ is the Grauert-Riemenschneider canonical sheaf on $X$, we have (see \cite{eins}, Theorem 2.1):
\begin{eqnarray}\label{eq:new2}
\mathcal{K}_X(F) = \pi_* \mathcal{K}_M(\pi^* F).
\end{eqnarray}
Takegoshi's vanishing theorem \cite{Ta} yields the vanishing of the higher direct image sheaves:
\begin{eqnarray}\label{eq:new3}
R^q \pi_* \mathcal{K}_M(\pi^* F) = 0\ ,\ q>0.
\end{eqnarray}
Moreover, square-integrable $(n,q)$-forms remain square-integrable under pull-back by $\pi$.
This pull-back commutes with the $\dq_w$-operator and it is continuous by \cite{eins}, (13).
The exceptional set of the resolution $\pi: M\rightarrow X$
does no harm as the $\dq$-equation in the $L^2$-sense extends over hypersurfaces.
So, $\pi$ induces a natural (continuous) mapping of complexes
\begin{eqnarray}\label{eq:new4}
\pi^*: \big(\mathcal{C}^{n,*}(F),\dq_w\big) \longrightarrow \big(\pi_* \mathcal{C}^{n,*}_\sigma(\pi^* F),\pi_* \dq_w\big),
\end{eqnarray}
and by Theorem \ref{thm:exactness1} and \eqref{eq:new3}, both complexes in \eqref{eq:new4}
are fine resolutions of $\mathcal{K}_X(F)=\pi_* \mathcal{K}_M(\pi^* F)$.
As $\pi^*$ commutes with the $\dq_w$-operator, it induces isomorphisms
\begin{eqnarray}\label{eq:new11}
H^{n,q}_{w,loc}\big(\Omega,F\big) &\overset{\cong}{\longrightarrow}& H^{n,q}_{w,loc}\big(\pi^{-1}(\Omega),\pi^* F\big) 
\cong H^q\big(\pi^{-1}(\Omega),\mathcal{K}_M(\pi^* F)\big),\\\label{eq:new12}
H^{n,q}_{w,cpt}\big(\Omega,F\big) &\overset{\cong}{\longrightarrow}& H^{n,q}_{w,cpt}\big(\pi^{-1}(\Omega),\pi^* F\big)
\cong H^q_{cpt}\big(\pi^{-1}(\Omega),\mathcal{K}_M(\pi^* F)\big)
\end{eqnarray}
for any open set $\Omega\subset X$ and $0\leq q\leq n$ (see \cite{eins}, Theorem 2.1).

\bigskip
For the rest of this section, assume that $X$ has only homogeneous isolated singularities.
In this situation, we have $\mathcal{K}^s_X \cong \mathcal{K}_X$
by \cite{eins}, Theorem 1.10 (with $D=\emptyset$), because homogeneous isolated singularities can be resolved by a single blow-up.

As the $\dq_s$-operator is stronger than the $\dq_w$-operator,
$\pi$ induces also a natural (continuous) mapping of complexes
\begin{eqnarray}\label{eq:new4b}
\pi^*: \big(\mathcal{F}^{n,*}(F),\dq_s\big) \longrightarrow \big(\pi_* \mathcal{C}^{n,*}_\sigma(\pi^* F),\pi_* \dq_w\big).
\end{eqnarray}
By Theorem \ref{thm:exactness2} and \eqref{eq:new3}, both complexes in \eqref{eq:new4b}
are fine resolutions of $\mathcal{K}^s_X(F)=\pi_* \mathcal{K}_M(\pi^* F)$
and so $\pi^*$ induces also isomorphisms
\begin{eqnarray}\label{eq:new11b}
H^{n,q}_{s,loc}\big(\Omega,F\big) &\overset{\cong}{\longrightarrow}& H^{n,q}_{w,loc}\big(\pi^{-1}(\Omega),\pi^* F\big) 
\cong H^q\big(\pi^{-1}(\Omega),\mathcal{K}_M(\pi^* F)\big),\\\label{eq:new12b}
H^{n,q}_{s,cpt}\big(\Omega,F\big) &\overset{\cong}{\longrightarrow}& H^{n,q}_{w,cpt}\big(\pi^{-1}(\Omega),\pi^* F\big)
\cong H^q_{cpt}\big(\pi^{-1}(\Omega),\mathcal{K}_M(\pi^* F)\big)
\end{eqnarray}
for any open set $\Omega\subset X$ and $0\leq q\leq n$.

\bigskip
\subsection{Metrizable topology of $L^{p,q}_{loc}(\Omega,F)$}

Let $\Omega\subset X$ be an open subset. We give $L^{p,q}_{loc}(\Omega,F)$ the structure of a Fr\'echet space
with the topology of $L^2$-convergence on compact subsets. This topology is obtained as follows.
Let $K_1 \subset K_2 \subset K_3 \subset ... \subset \Omega$ be a compact exhaustion of $\Omega$,
and define the separating family of seminorms
\begin{eqnarray}\label{eq:seminorms}
p_j(\eta) := \left( \int_{K_j^*} |\eta|_F^2 dV_X \right)^{1/2}
\end{eqnarray}
for $\eta\in L_{loc}^{p,q}(\Omega,F)$ and  $j=1, 2, ...$, where $dV_X$ is the volume form on $X^*$ induced by the Hermitian metric
of the Hermitian space $X$. $L_{loc}^{p,q}(\Omega,F)$ is then a Fr\'echet space with the metric
$$d(\eta,\omega) := \sum_{j=1}^\infty 2^{-j} \frac{p_j(\eta-\omega)}{1+p_j(\eta-\omega)}\ \ ,\ \eta,\omega\in L_{loc}^{p,q}(\Omega,F)$$
(compare e.g. \cite{Rd}, Theorem 1.37, Remark 1.38 and Example 1.44).
The induced topology is also called the topology of compact $L^2$-convergence.
It is not hard to see that this topology does not depend on the compact exhaustion.

\medskip
\subsection{Dual space of $L_{loc}^{p,q}(\Omega,F)$}

Let us consider the vector space of compactly supported $L^2$-forms with values in the (Hermitian) dual bundle $F^*$,
\begin{eqnarray*}
L_{cpt}^{r,s}(\Omega,F^*):= \{ \eta\in L^{r,s}_{loc}(\Omega,F^*): \supp \eta \subset\subset \Omega\},
\end{eqnarray*}
which inherits a metric as a subspace of $L_{loc}^{r,s}(\Omega,F^*)$. Unfortunately,
this metric is not complete, i.e. $L_{cpt}^{r,s}(\Omega,F^*)$ is not a closed subspace of $L_{loc}^{r,s}(\Omega,F^*)$.

However, analogously to the Schwartz topology on spaces of test-forms,
we can give $L_{cpt}^{r,s}(\Omega,F^*)$ the structure of a complete locally convex topological vector space as follows.
For a fixed compact set $K\subset \Omega$, let
$$\mathcal{D}_K^{r,s}:=L_{K}^{r,s}(\Omega,F^*):=\{ \eta\in L^{r,s}_{loc}(\Omega,F^*): \supp \eta\subset K\},$$
carrying the induced Fr\'echet space structure, and let $\tau_K$ be the topology on $\mathcal{D}_K^{r,s}$.
Following \cite{Rd}, Definition 6.3, let $\beta$ be the collection of all convex balanced sets $W\subset L_{cpt}^{r,s}(\Omega,F^*)$
such that $\mathcal{D}_K^{r,s}\cap W \in \tau_K$ for every compact set $K\subset \Omega$,
and define the topology $\tau$ on $L_{cpt}^{r,s}(\Omega,F^*)$ as the collection of all unions of sets of the form $\phi + W$
with $\phi\in L_{cpt}^{r,s}(\Omega,F^*)$ and $W\in\beta$. 

This definition means nothing else but that $L_{cpt}^{r,s}(\Omega,F^*)$ is the (topological) inductive limit of the Frech\'et spaces $\mathcal{D}_K^{r,s}$.
So, $L^{r,s}_{cpt}(\Omega,F^*)$ is an (LF)-space in the sense of Dieudonn\'e-Schwartz \cite{DS} (consider e.g. their first example of an (LF)-space).

\begin{thm}\label{thm:rudin}
$\tau$ is a topology, making $L_{cpt}^{r,s}(\Omega,F^*)$ into an (LF)-space.
Particularly, $L_{cpt}^{r,s}(\Omega,F^*)$ is a locally convex topological vector space and every Cauchy sequence converges.
For a sequence of forms $\{\phi_k\}_{k\geq 1} \subset L_{cpt}^{r,s}(\Omega,F^*)$,
we have $\phi_k \rightarrow 0$ in the topology $\tau$ exactly if 
\begin{itemize}
\item there exists a compact set $K\subset \Omega$ such that $\supp \phi_k \subset K$ for all $k\geq 1$, and 
\item $\phi_k\rightarrow 0$ in $L_{K}^{r,s}(\Omega,F^*)$.
\end{itemize}
\end{thm}

\begin{proof}
Is contained in \cite{DS}, Section 3 and Section 4.
Alternatively, see also \cite{Rd}, Theorem 6.4 and Theorem 6.5.
\end{proof}


We can now show:

\begin{thm}\label{thm:duality}
Under the non-degenrate pairing
\begin{eqnarray}\label{eq:pairing}
L_{loc}^{p,q}(\Omega,F) \times L_{cpt}^{n-p,n-q}(\Omega,F^*) \rightarrow \C\ ,\ \ (\eta,\omega) \mapsto \int_{\Omega^*} \eta\wedge\omega,
\end{eqnarray}
$L_{cpt}^{n-p,n-q}(\Omega,F^*)$ is algebraically isomorphic to the dual space of $L_{loc}^{p,q}(\Omega,F)$,
and, vice versa, $L_{loc}^{p,q}(\Omega,F)$ is algebraically isomorphic to the dual space of $L_{cpt}^{n-p,n-q}(\Omega,F^*)$.
\end{thm}

\smallskip
By the dual space $V'$ of a topological vector space $V$,
we understand the vector space of continuous linear forms on $V$.
The dual space can carry different topologies (the two most important are the weak and the strong topology),
but we do not need to discuss such issues here as we are concerned in Theorem \ref{thm:duality}
only with algebraic isomorphy.

\smallskip
\begin{proof}
It is clear that the pairing \eqref{eq:pairing} is well-defined and non-degenerate.

\medskip
(I) Let $\omega\in L_{cpt}^{n-p,n-q}(\Omega,F^*)$. Then $\omega$ defines a continuous linear
functional $\Lambda_\omega$ on $L_{loc}^{p,q}(\Omega,F)$ by
\begin{eqnarray*}
\Lambda_\omega: L_{loc}^{p,q}(\Omega,F) \rightarrow \C\ , \ \eta \mapsto \int_{\Omega^*} \eta\wedge\omega.
\end{eqnarray*}
Continuity of $\Lambda_\omega$ is easy to check by use of a separating family of semi-norms.
So, we obtain an injective linear map
\begin{eqnarray*}
\Phi: && L_{cpt}^{n-p,n-q}(\Omega,F^*) \longrightarrow \big( L_{loc}^{p,q}(\Omega,F) \big)'.
\end{eqnarray*}
To show that $\Phi$ is also surjective, let $\Lambda \in \big( L_{loc}^{p,q}(\Omega,F) \big)'$.
We will first prove that $\Lambda$ has compact support.
Recall that the metric of $L_{loc}^{p,q}(\Omega,F)$
is induced by the separating semi-norms $p_j$ in the sense of \cite{Rd}, Theorem 1.37 (see \eqref{eq:seminorms}).
Assign to each $p_j$ and each positive integer $n$ the set
$$V(j,n):= \{ \eta: p_j(\eta)< 1/n \}.$$
Then the collection of all finite intersections of the sets $V(j,n)$ is a convex balanced local base for the topology
of $L_{loc}^{p,q}(\Omega,F)$ (see \cite{Rd}, Theorem 1.37).
As $p_1(\eta) \leq p_2(\eta) \leq p_3(\eta)\leq ...$ in our situation,
already the collection of the sets $V(j,n)$ is a convex local base for the topology.
So, by continuity of $\Lambda$, there exists indices $j_0$ and $n_0$ such that
$$|\Lambda(\eta)| \leq 1 \ \ \mbox{ for all } \eta\in V(j_0,n_0).$$
But then 
$$|\Lambda (\eta)| \leq n_0 p_{j_0}(\eta) = n_0 \left( \int_{K_{j_0}^*} |\eta|^2_F dV_X\right)^{1/2}\ \ \forall \eta \in L_{loc}^{p,q}(\Omega,F).$$
Thus $\Lambda$ must have compact support on $K:=K_{j_0}$, i.e. $\Lambda(\eta) = \Lambda(\eta|_K)$.

On the other hand, by trivial extension, we have a continuous inclusion
$$L^{p,q}(K,F) \subset L_{loc}^{p,q}(\Omega,F),$$
where $L^{p,q}(K,F)$ carries the usual $L^2$-Hilbert space topology
(a sequence converging in $L^{p,q}(K,F)$ is, after trivial extension, also converging in $L^{p,q}_{loc}(\Omega,F)$).
So, $\Lambda$ is also a continuous linear functional on $L^{p,q}(K,F)$
and thus represented on $K^*$ by an $L^2$-form $\omega_K$.
But as $\Lambda$ has support in $K$, this means that $\Lambda$ is represented by $\omega_K$ on all of $\Omega$
by extending $\omega_K$ trivially to $\Omega^*$.

\medskip
(II) Let $\eta\in L_{loc}^{p,q}(\Omega,F)$.
Then it is easy to see that $\eta$ defines a continuous linear functional
\begin{eqnarray*}
\Lambda_\eta: L_{cpt}^{n-p,n-q}(\Omega,F^*) \rightarrow \C\ , \ \omega \mapsto \int_{\Omega^*} \eta\wedge\omega.
\end{eqnarray*}
For continuity, note the following property of an (LF)-space $E$ with defining sequence $\{E_j\}_j$:
a linear mapping from $E$ into a locally convex topological vector space $F$ is continuous if and only if
its restriction to each $E_j$ is continuous (see e.g. \cite{DS}, Proposition 5).
So, we obtain an injective linear mapping
\begin{eqnarray*}
\Psi: && L_{loc}^{p,q}(\Omega,F) \longrightarrow \big( L_{cpt}^{n-p,n-q}(\Omega,F^*) \big)'\ ,\ \eta\mapsto \Lambda_\eta.
\end{eqnarray*}
To show that $\Psi$ is also surjective, consider $\Lambda \in \big( L_{cpt}^{n-p,n-q}(\Omega,F^*) \big)'$.
But, by trivial extension,
$$L^{n-p,n-q}(V^*,F^*) \cong \mathcal{D}^{n-p,n-q}_V = L_{V}^{n-p,n-q}(\Omega,F^*) \subset L_{cpt}^{n-p,n-q}(\Omega,F^*)$$
as topological subspaces for any compact subset $V\subset\subset \Omega$.
So, $\Lambda$ is represented by an $L^2$-form $\eta_V$ on any such set $V$.
But the $\eta_V$ must coincide where their domains intersect,
so that $\Lambda$ is represented by a globally defined form $\eta\in L_{loc}^{p,q}(\Omega,F)$.
\end{proof}

\subsection{$L^2$-Hilbert space duality}\label{ssec:l2serre}

We denote by $*$ the Hodge-$*$-operator on the complex Hermitian manifold $X^*=X\setminus\Sing X$.
It is convenient to work with the conjugate-linear operator 
$$\o{*} \eta := * \o{\eta}.$$

Let
$\tau: F \rightarrow F^*$
be the canonical conjugate-linear bundle isomorphism of $F$ onto its dual bundle.
We can then define the conjugate-linear isomorphism
$$\o{*}_{F}: \Lambda^{p,q} T^*M\otimes F \rightarrow \Lambda^{n-p,n-q} T^*M \otimes F^*$$
by setting $\o{*}_{F}( \eta\otimes e) := \o{*} \eta \otimes \tau(e)$.
This gives the following representation for
the inner product on $(p,q)$-forms with values in $F$ on an open set $\Omega\subset X$:
\begin{eqnarray}\label{eq:ip1}
(\eta,\psi)_{F,\Omega} &=& \int_{\Omega^*} \langle\eta,\psi\rangle_{F} dV_X = \int_{\Omega^*} \eta\wedge \o{*}_{F} \psi,\\
\|\eta\|_{F,\Omega} &=& \sqrt{(\eta,\eta)_{F,\Omega}}.\label{eq:ip2}
\end{eqnarray}
Suppose that $\eta$, $\psi$ are smooth forms with values in $F$ and compact support in $\Omega^*$,
$\eta$ of degree $(p,q-1)$ and $\psi$ of degree $(p,q)$.
Then it is easy to compute by Stokes' Theorem that:
\begin{eqnarray*}
(\dq\eta,\psi)_{F,\Omega} = (-1)^{p+q}\int_{\Omega^*} \eta\wedge\dq\o{*}_{F} \psi
= -\int_{\Omega^*} \eta\wedge \o{*}_{F}\o{*}_{F^*} \dq \o{*}_{F}\psi
= (\eta, -\o{*}_{F^*}\dq \o{*}_{F}\psi)_{F,\Omega}.
\end{eqnarray*}
Thus, we note (cf. e.g. \cite{eins}, Lemma 2.2):

\begin{lem}\label{lem:theta}
The formal adjoint of the $\dq$-operator for forms with values in the Hermitian holomorphic line bundle $F$
with respect to the $\|\cdot\|_{F}$-norm is
\begin{eqnarray}\label{eq:theta}
\theta := - \o{*}_{F^*} \dq \o{*}_{F}.
\end{eqnarray}
\end{lem}
Let 
$$\dq_{cpt}: A^{p,q}_{cpt}(\Omega^*,F) \rightarrow A^{p,q+1}_{cpt}(\Omega^*,F)$$
be the $\dq$-operator on smooth $F$-valued forms with compact support in $\Omega^*$.
Then we denote by
$$\dq_{max}: L^{p,q}(\Omega^*,F) \rightarrow L^{p,q+1}(\Omega^*,F)$$
the maximal and by
$$\dq_{min}: L^{p,q}(\Omega^*,F) \rightarrow L^{p,q+1}(\Omega^*,F)$$
the minimal closed Hilbert space extension of the operator $\dq_{cpt}$
as densely defined operator from $L^{p,q}(\Omega^*,F)$ to $L^{p,q+1}(\Omega^*,F)$.

For $F$-valued forms, 
let $H^{p,q}_{max}(\Omega^*,F)$ be the $L^2$-Dolbeault cohomology on $\Omega^*$ with respect
to the maximal closed extension $\dq_{max}$, i.e. the $\dq$-operator in the sense of distributions on $\Omega^*$,
and $H^{p,q}_{min}(\Omega^*,F)$ the $L^2$-Dolbeault cohomology with respect to the minimal closed
extension $\dq_{min}$.

Note that the $\dq$-operator in the sense of distributions, $\dq_{max}$, agrees with the operator $\dq_w$
defined above restricted to $L^{p,q}(\Omega^*,F)$.
On the other hand, $\dq_{min}$ has the same boundary condition as $\dq_s$ at the singular set $\Sing X$,
but it comes also with a Dirichlet boundary condition at $b\Omega$ which does not appear for $\dq_s$.

We will now identify the Hilbert space adjoints $\dq_{max}^*$ and $\dq_{min}^*$ of $\dq_{max}$ and $\dq_{min}$.
Let $\theta$ be the formal adjoint of $\dq$ as computed in Lemma \ref{lem:theta},
and denote by $\theta_{cpt}$ its action on smooth $F$-valued forms compactly supported in $\Omega^*$:
$$\theta_{cpt}: A^{p,q}_{cpt}(\Omega^*,F) \rightarrow A^{p,q-1}_{cpt}(\Omega^*,F).$$
This operator is graph closable as an operator $L^2_{p,q}(\Omega^*,F)\rightarrow L^2_{p,q-1}(\Omega^*,F)$,
and as for the $\dq$-operator, we denote by $\theta_{min}$ its minimal closed extension, i.e.
the closure of the graph, and by $\theta_{max}$ the maximal closed extension, that is the
$\theta$-operator in the sense of distributions with respect to compact subsets of $\Omega^*$.
By \eqref{eq:theta} we have:
\begin{eqnarray*}
\theta_{min} = - \o{*}_{F^*} \dq_{min} \o{*}_{F}\ \ ,\ \ \theta_{max} = - \o{*}_{F^*} \dq_{max} \o{*}_{F}.
\end{eqnarray*}
By definition, $\dq_{max} = \theta_{cpt}^*$,
and it follows that
\begin{eqnarray}\label{eq:adjoint1}
\dq_{max}^* = \big(\theta_{cpt}^{*}\big)^* = \o{\theta_{cpt}} = \theta_{min} = -\o{*}_{F^*}\dq_{min}\o{*}_{F},
\end{eqnarray}
where we denote by $\o{\theta_{cpt}}$ also the closure of the graph of $\theta_{cpt}$.
Analogously, $\theta_{max}=\dq_{cpt}^*$ (by definition) implies 
\begin{eqnarray}\label{eq:adjoint2}
\dq_{min}^* = \theta_{max} = - \o{*}_{F^*} \dq_{max} \o{*}_{F}.
\end{eqnarray}
For the sake of completeness, let us recall (see \cite{eins}, Theorem 2.3):

\begin{thm}\label{thm:serre-duality}
Assume that the $\dq$-operators in the sense of distributions
\begin{eqnarray*}
\dq_{max}: L^{p,q-1}(\Omega^*,F) \rightarrow L^{p,q}(\Omega^*,F)\ \ ,\ \ \dq_{max}: L^{p,q}(\Omega^*,F) \rightarrow L^{p,q+1}(\Omega^*,F)
\end{eqnarray*}
both have closed range (with the usual assumptions for $q=0$ or $q=n$). 
Then there exists a non-degenerate pairing
$$\{\cdot,\cdot\}: H^{p,q}_{max}(\Omega^*,F) \times H^{n-p,n-q}_{min}(\Omega^*,F^*) \rightarrow \C$$
given by $\{[\eta],[\psi]\}:=\int_{\Omega^*} \eta\wedge\psi$,
inducing (topological) isomorphisms $H^{p,q}_{max}(\Omega^*,F) \cong \big(H^{n-p,n-q}_{min}(\Omega^*,F^*)\big)'$
and $H^{n-p,n-q}_{min}(\Omega^*,F^*) \cong \big(H^{p,q}_{max}(\Omega^*,F)\big)'$.
\end{thm}

\subsection{Fr\'echet sheaves}\label{ssec:fs}

For convenience of the reader, let us recall from \cite{GR} a few preliminaries on the (unique) Fr\'echet space structure 
on coherent analytic sheaves. 

\begin{defn}[\cite{GR}, Definition VIII.A.3]\label{defn:fs}
Let $\mathcal{S}$ be a sheaf of vector spaces over a topological space $X$.
$\mathcal{S}$ is a {\bf Fr\'echet sheaf} if there is a neighborhood basis $\mathcal{U}=\{U\}$ of open sets 
such that the following two conditions hold:
\begin{itemize}
\item[(1)] $H^0(U,\mathcal{S})$ can be given the structure of a Fr\'echet space for all $U\in\mathcal{U}$.
\item[(2)] If $U,V\in\mathcal{U}$ and $V\subset U$, then the restriction map
$r_{UV}: H^0(U,\mathcal{S}) \rightarrow H^0(V,\mathcal{S})$ is continuous.
\end{itemize}
\end{defn}

\begin{thm}[\cite{GR}, Theorem VIII.A.7]\label{thm:fs1}
Let $(X,\OO_X)$ be a reduced complex space. There is a unique way of making every coherent analytic sheaf into a Fr\'echet sheaf
so that the following two conditions are satisfied.
\begin{itemize}
\item[(1)] If $\mathcal{S}$ is a subsheaf of $(\OO_X)^N$, the space of sections of 
$\mathcal{S}$ has the topology of uniform convergence on compact subsets.
\item[(2)] For any two coherent sheaves $\mathcal{S}$, $\mathcal{T}$ there is a neighborhood basis $\mathcal{U}$
such that $H^0(U,\mathcal{S})$, $H^0(U,\mathcal{T})$ are Fr\'echet spaces for all $U\in\mathcal{U}$.
Any $\OO_X$-homomorphism $\varphi: \mathcal{S} \rightarrow \mathcal{T}$ is continuous.
\end{itemize}
\end{thm}

\begin{thm}[\cite{GR}, Theorem VIII.A.8]\label{thm:fs2}
Let $(X,\OO_X)$ be a reduced complex space. Let $\mathcal{S}$ be a coherent analytic sheaf 
(and thus by Theorem \ref{thm:fs1} a Fr\'echet sheaf). For any open set $W$ of $X$, $H^0(W,\mathcal{S})$
can be given the structure of a Fr\'echet space in a unique way so that the restriction mappings are continuous.
We can choose a family of pseudonorms $\|\cdot\|_K$ on $H^0(K,\mathcal{S})$ for all compact sets $K$, so that,
if $K\subset K'$ and $f\in H^0(K',\mathcal{S})$ then $\|f\|_K \leq \|f\|_{K'}$. The topology of $H^0(W,\mathcal{S})$ is that
determined by the pseudonorms $\|\cdot\|_K$ for all $K\subset W$.
\end{thm}

\bigskip
\section{Topological Serre duality for $\dq$-operators}\label{sec:serre}

We will now derive duality statements similar to $L^2$-Hilbert space duality, Theorem \ref{thm:serre-duality},
for the operators $\dq_w$ and $\dq_s$ on the spaces $L^{p,q}_{loc}$ and $L^{p,q}_{cpt}$, respectively.
We have to face the problem that we deal just with densely defined operators on locally convex topological vector spaces.
This requires some extra work.
Moreover, another main difficulty is to show 
that the operators $\dq_w$ and $\dq_s$ are topologically dual at arbitrary singularities.

\medskip
\subsection{Separated cohomology groups}

Note that the cohomology groups defined in Section \ref{ssec:complexes}, $H^{p,q}_{w/s,loc/cpt}(\Omega,F)$
are Hausdorff, i.e., separated, if and only if the corresponding $\dq$-operator has closed range on $\Omega$.

In view of duality statements, we need to work with the associated separated cohomology groups
which are obtained by taking the quotient spaces by $\o{\{0\}}$, i.e, considering the quotient spaces
$\ker \dq_{w/s} / \overline{\Im \dq_{w/s}}$.
We call these the separated cohomology groups, denoted by
\begin{eqnarray*}
\o{H}^{p,q}_{w/s,loc/cpt} (\Omega,F) := \frac{H^{p,q}_{w/s,loc/cpt}(\Omega,F)}{ \o{\{0\}} }
= \frac{\ker \dq_{w/s} }{ \overline{\Im \dq_{w/s}}}.
\end{eqnarray*}

\medskip
\subsection{Duality between separated $\dq_w$- and $\dq_s$-cohomology}

\begin{thm}
\label{thm:duality1}
Under the non-degenerate pairing
\begin{eqnarray}\label{eq:pairing1}
\o{H}^{p,q}_{w,loc}(\Omega,F) \times \o{H}^{n-p,n-q}_{s,cpt}(\Omega,F^*) \rightarrow \C\ ,
\ \ ([\eta],[\omega]) \mapsto \int_{\Omega^*} \eta\wedge\omega,
\end{eqnarray}
$\o{H}^{n-p,n-q}_{s,cpt}(\Omega,F^*)$ is algebraically isomorphic
to the dual space of $\o{H}^{p,q}_{w,loc}(\Omega,F)$,
and, vice versa, $\o{H}^{p,q}_{w,loc}(\Omega,F)$ is algebraically isomorphic
to the dual space of $\o{H}^{n-p,n-q}_{s,cpt}(\Omega,F^*)$.
\end{thm}

\begin{proof}
(I) The integral in \eqref{eq:pairing1} is always finite as $\eta\in L^{p,q}_{loc}$ and $\omega\in L^{n-p,n-q}_{cpt}$.
Assume that $\eta= \dq_w \eta'$ with $\eta' \in \Dom \dq_w \cap L^{p,q-1}_{loc}$. Then
\begin{eqnarray}\label{eq:rev11}
\int_{\Omega^*} \eta\wedge \omega = \int_{\Omega^*} \dq_w \eta' \wedge \omega = \int_{\Omega^*} \eta'\wedge \dq_s \omega = 0,
\end{eqnarray}
because $\omega$ has compact support in $\Omega$ and can be (by definition of $\dq_s$) approximated in the graph norm 
by forms with compact support away from the singular set so that partial integration is possible.
By approximation, \eqref{eq:rev11} holds as well for $\eta \in \o{\Im \dq_w}$.

Analogously, assume that $\omega=\dq_s \omega'$ with $\omega' \in \Dom\dq_s \cap L^{n-p,n-q-1}_{cpt}$.
Then
\begin{eqnarray}\label{eq:rev12}
\int_{\Omega^*} \eta\wedge \omega = \int_{\Omega^*} \eta \wedge \dq_w \omega' = \int_{\Omega^*} \dq_s \eta\wedge \omega' = 0
\end{eqnarray}
by the same argument fas above.
By approximation, \eqref{eq:rev12} holds as well for $\omega \in \o{\Im \dq_s} \cap L^{n-p,n-q}_{cpt}$.
This shows that the pairing \eqref{eq:pairing1} is well-defined.

\medskip
(II) Let $[\omega] \in \o{H}^{n-p,n-q}_{s,cpt}(\Omega,F^*)$,
represented by $\omega\in \ker\dq_s \cap L^{n-p,n-q}_{cpt}(\Omega,F^*)$.
Then $\omega$ defines by Theorem \ref{thm:duality} a continuous linear functional on $L^{p,q}_{w,loc}(\Omega,F)$ by
the assignment
$$\eta  \mapsto \int_{\Omega^*} \eta\wedge\omega.$$

This induces (by the partial integration argument from above) the continuous\footnote{
Let $X$ be a topological vector space, $N$ a closed subspace and $X/N$ the quotient space with the quotient topology.
Then the projection $\pi: X \rightarrow X/N$ is an open mapping (see e.g. \cite{Rd}, Theorem 1.41).
So, if $\Lambda$ is a continuous linear functional on $X$ with $N\subset \ker \Lambda$, then $\Lambda$ induces a continuous
linear functional on $X/N$.} linear functional 
$$[\eta] \mapsto \int_{\Omega^*} \eta\wedge \omega$$
on the quotient space $\o{H}^{p,q}_{w,loc}(\Omega,F))$ of the closed subspace $\ker \dq_w \cap L^{p,q}_{w,loc}(\Omega,F)$.
Thus, $[\omega]$ represents in fact a continuous linear functional in $\big(\o{H}^{p,q}_{w,loc}(\Omega,F)\big)'$

Conversely, let $\Lambda \in \big(\o{H}^{p,q}_{w,loc}(\Omega,F)\big)'$.
As we consider the separated cohomology, the projection
$$\pi: \ker \dq_w \rightarrow \frac{\ker \dq_w}{\o{\Im\dq_w}} = \o{H}^{p,q}_{w,loc}\big(\Omega,F\big)$$
is continuous (see e.g. \cite{Rd}, Theorem 1.41), and so $\Lambda\circ \pi$ is a continuous linear functional on $\ker\dq_w$.
But $\ker\dq_w$ is a closed subspace of $L^{p,q}_{w,loc}(\Omega,F)$.
So, $\Lambda\circ\pi$ extends by the Hahn-Banach theorem to a continuous linear functional $\Lambda'$ on $L^{p,q}_{loc}(\Omega,F)$.
But we know already that the dual space of $L^{p,q}_{loc}(\Omega,F)$ is isomorphic to $L^{n-p,n-q}_{cpt}(\Omega,F^*)$.
Thus, there is by Theorem \ref{thm:duality} a form $\omega \in L^{n-p,n-q}_{cpt}(\Omega,F^*)$ representing $\Lambda'$:
\begin{eqnarray*}
\Lambda': L^{p,q}_{loc}(\Omega,F) &\rightarrow& \C\ ,\\
f &\mapsto& \Lambda'\big(f\big)=\int_{\Omega^*} f\wedge \omega.
\end{eqnarray*}
So, the continuous linear functional $\Lambda\circ\pi$ is represented by $\omega$, too:
\begin{eqnarray}\label{eq:duality01}
\Lambda\circ\pi:\ \ker \dq_w \cap L^{p,q}_{loc}(\Omega,F) \rightarrow \C\ ,\ \eta \mapsto \int_{\Omega^*} \eta\wedge \omega.
\end{eqnarray}

We claim that $\omega \in \Dom\dq_s$ and $\dq_s \omega=0$.
But \eqref{eq:duality01} implies that
\begin{eqnarray*}
(\dq_w g, \o{*}_{F^*}\omega)_{F,\Omega^*} = \pm \int_{\Omega^*} \dq_w g\wedge \omega = 0
\end{eqnarray*}
for all $g\in \Dom\dq_w \cap L^{p,q-1}_{loc}(\Omega,F)$, because $\Lambda\circ\pi$ vanishes on $\o{\Im \dq_w}$.
But then, particularly, 
\begin{eqnarray}\label{eq:duality02}
(\dq_{max} g,\o{*}_{F^*}\omega)_{\Omega^*}=0\ \ \ \forall g\in \Dom\dq_{max} \cap L^{p,q-1}(\Omega^*,F).
\end{eqnarray}
Now recall that $\dq_{max}^* = \theta_{min}$ (see \eqref{eq:adjoint1}).
Thus, \eqref{eq:duality02} just means that $\omega \in \Dom \dq_{min} \cap L^{n-p,n-q}(\Omega^*,F^*)$
with $\dq_{min}\omega=0$.
But then also $\omega \in \Dom\dq_s$ with $\dq_s\omega=0$.
This shows that in fact any continuous linear functional $\Lambda \in \big(\o{H}^{p,q}_{w,loc}(\Omega,F)\big)'$
is represented by a cohomology class $[\omega] \in \o{H}^{n-p,n-q}_{s,cpt}(\Omega,F^*)$
under the pairing \eqref{eq:pairing1}.

\bigskip
(III) It remains to show that, conversely,
$\o{H}^{p,q}_{w,loc}(\Omega,F)$ is the topological dual of $\o{H}^{n-p,n-q}_{s,cpt}(\Omega,F^*)$.
As above, it is clear that $[\eta]\in \o{H}^{p,q}_{w,loc}(\Omega,F)$ represents a continuous linear functional in
$\big(\o{H}^{n-p,n-q}_{s,cpt}(\Omega,F^*)\big)'$.

So, consider $\Lambda \in \big(\o{H}^{n-p,n-q}_{s,cpt}(\Omega,F^*)\big)'$.
As above, using the continuous projection $\pi: \ker \dq_s \rightarrow \ker\dq_s/ \o{\Im\dq_s}$,
we obtain the continuous linear functional $\Lambda\circ \pi$ on $\ker\dq_s$
which is a closed subspace of the locally convex topological vector space $L^{n-p,n-q}_{cpt}(\Omega,F^*)$.
So, by the Hahn-Banach theorem (see e.g. \cite{Rd}, Theorem 3.6), there is an extension of $\Lambda\circ\pi$
to a linear continuous functional $\Lambda'$ on $L^{n-p,n-q}_{cpt}(\Omega,F^*)$
which is then by Theorem \ref{thm:duality} represented by a form $\eta \in L^{p,q}_{loc}(\Omega,F)$:
\begin{eqnarray*}
\Lambda': L^{n-p,n-q}_{cpt}(\Omega,F^*) &\rightarrow& \C\ ,\\
f &\mapsto& \Lambda'\big(f\big)=\int_{\Omega^*} f\wedge \eta.
\end{eqnarray*}
So, the continuous linear funtional $\Lambda\circ\pi$ is represented by $\eta$, too:
\begin{eqnarray}\label{eq:duality03}
\Lambda\circ\pi:\ L^{n-p,n-q}_{cpt}(\Omega,F^*) \rightarrow \C\ ,\ \omega \mapsto \int_{\Omega^*} \omega\wedge \eta.
\end{eqnarray}
We claim that $\eta \in \Dom\dq_w$ and $\dq_w \eta=0$, i.e. $\dq\eta=0$ in the sense of distributions on $\Omega^*$.
But this follows from the following observation. \eqref{eq:duality03} implies particularly that
\begin{eqnarray}\label{eq:duality04}
\int_{\Omega^*} \dq\varphi\wedge \eta=0
\end{eqnarray}
for any smooth testform $\varphi\in A^{n-p,n-q-1}_{cpt}(\Omega^*,F^*)$,
because then $\varphi \in \Dom\dq_s \cap L^{n-p,n-q-1}_{cpt}(\Omega^*,F^*)$
by \cite{PAMS}, Theorem 1.6, and $\Lambda\circ\pi$ vanishes on $\o{\Im\dq_s}$.

This shows that in fact any continuous linear functional $\Lambda \in \big(\o{H}^{n-p,n-q}_{s,cpt}(\Omega,F^*)\big)'$
is represented by a cohomology class $[\eta] \in \o{H}^{p,q}_{w,loc}(\Omega,F)$
under the pairing \eqref{eq:pairing1}.
\end{proof}

There is another interesting topological duality pairing:

\begin{thm}
\label{thm:duality2}
Under the non-degenerate pairing
\begin{eqnarray}\label{eq:pairing2}
\o{H}^{p,q}_{s,loc}(\Omega,F) \times \o{H}^{n-p,n-q}_{w,cpt}(\Omega,F^*) \rightarrow \C\ ,
\ \ ([\eta],[\omega]) \mapsto \int_{\Omega^*} \eta\wedge\omega,
\end{eqnarray}
$\o{H}^{n-p,n-q}_{w,cpt}(\Omega,F^*)$ is algebraically isomorphic
to the dual space of $\o{H}^{p,q}_{s,loc}(\Omega,F)$,
and, vice versa, $\o{H}^{p,q}_{s,loc}(\Omega,F)$ is algebraically isomorphic
to the dual space of $\o{H}^{n-p,n-q}_{w,cpt}(\Omega,F^*)$.
\end{thm}

The proof is similar to the proof of Theorem \ref{thm:duality1}, but there is an additional difficulty that we should discuss carefully.

\begin{proof}
(I) First, it is seen as in the proof of Theorem \ref{thm:duality1} that the pairing \eqref{eq:pairing2} is well-defined
because partial integration is possible: forms in $\ker\dq_s \cap L^{p,q}_{loc}(\Omega,F)$
can be approximated in the graph norm by forms with support away
from the singular set, and forms in $L^{n-p,n-q}_{w,cpt}(\Omega,F^*)$ have compact support in $\Omega$.

\medskip
(II) As in the proof of Theorem \ref{thm:duality1}, a cohomology class in $ \o{H}^{n-p,n-q}_{w,cpt}(\Omega,F^*)$
defines a continuous linear functional on $\o{H}^{p,q}_{s,loc}(\Omega,F)$.

Conversely, let $\Lambda \in \big(\o{H}^{p,q}_{s,loc}(\Omega,F)\big)'$.
It is seen completely analogous to the proof of Theorem \ref{thm:duality1} that $\Lambda\circ\pi$ is represented
by a form $\omega \in L^{n-p,n-q}_{cpt}(\Omega,F^*)$:
\begin{eqnarray}\label{eq:duality05}
\Lambda\circ\pi:\ \ker\dq_s\cap L^{p,q}_{loc}(\Omega,F) \rightarrow \C\ ,\ \eta \mapsto \int_{\Omega^*} \eta\wedge \omega.
\end{eqnarray}
We claim that $\omega \in \Dom\dq_w$ and $\dq_w \omega=0$.
But this is the same as above: \eqref{eq:duality05} yields particularly (cf. \eqref{eq:duality04}) that
\begin{eqnarray}\label{eq:duality06}
\int_{\Omega^*} \dq\varphi\wedge \omega=0\ \ \ \forall \varphi\in A^{p,q-1}_{cpt}(\Omega^*,F).
\end{eqnarray}

\medskip
(III) It remains to show that $\o{H}^{p,q}_{s,loc}(\Omega,F)$ 
is the topological dual of $\o{H}^{n-p,n-q}_{w,cpt}(\Omega,F^*)$.
Here, an additional difficulty appears.

Again, it is clear that a cohomology class in $\o{H}^{p,q}_{s,loc}(\Omega,F)$
represents a continuous linear functional on $\o{H}^{n-p,n-q}_{w,cpt}(\Omega,F^*)$.

For the converse, let $\Lambda\in \big(\o{H}^{n-p,n-q}_{w,cpt}(\Omega,F^*)\big)'$.
It is seen as above that $\Lambda\circ\pi$ is represented
by a form $\eta \in L^{p,q}_{loc}(\Omega,F)$:
\begin{eqnarray}\label{eq:duality07}
\Lambda\circ\pi:\ \ker\dq_w \cap L^{n-p,n-q}_{cpt}(\Omega,F^*)\big) \rightarrow \C\ ,
\ \omega \mapsto \int_{\Omega^*} \omega\wedge \eta.
\end{eqnarray}
We have to show that $\eta \in\Dom\dq_s$ with $\dq_s\eta=0$.
As above, it follows from \eqref{eq:duality07} that 
\begin{eqnarray}\label{eq:duality08}
\int_{\Omega^*} \dq\varphi\wedge \eta=0\ \ \ \forall \varphi\in A^{n-p,n-q-1}_{cpt}(\Omega^*,F^*).
\end{eqnarray}
Thus $\eta\in\Dom\dq_w$ and $\dq_w \eta=0$, i.e. $\dq\eta=0$ in the sense of distributions on $\Omega^*$.

\medskip
It remains to show that $\eta\in \Dom\dq_s$. That has to be checked on compact subsets $K\subset\subset \Omega$.
So, let $K\subset\subset \Omega$ compact and
let $\chi\in C^\infty_{cpt}(\Omega)$ be a smooth cut-off function with compact support in $\Omega$
which is identically $1$ in a neighborhood of $K$.
Then $\chi$ and $\dq\chi$ are uniformly bounded (in the sup-norm).
So, it follows from \eqref{eq:duality07} that
\begin{eqnarray}\label{eq:duality09}
\int_{\Omega^*} \dq_w(\chi \varphi)\wedge \eta =0
\end{eqnarray}
for all $\varphi \in \Dom\dq_w \cap L^{n-p,n-q-1}(\Omega^*,F^*)$
because then we have $\chi\varphi \in \Dom\dq_w \cap L^{n-p,n-q-1}_{cpt}(\Omega,F^*)$
and $\dq(\chi\varphi) \in \Im\dq_w \cap L^{n-p,n-q}_{cpt}(\Omega,F^*)\big)$.

But now it follows from \eqref{eq:duality09} that
\begin{eqnarray}\label{eq:duality10}
\big(\dq_w \varphi, \o{*}_{F} (\chi \eta)\big)_{\Omega^*,F^*} &=& \pm \int_{\Omega^*} \dq_w \varphi\wedge \chi \eta 
= \pm \int_{\Omega^*} \varphi\wedge \dq\chi\wedge \eta\\
&=& \pm \big(\varphi, \o{*}_F (\dq\chi\wedge \eta)\big)_{\Omega^*,F^*}\label{eq:duality11}
\end{eqnarray}
for all $\varphi \in \Dom\dq_w \cap L^{n-p,n-q-1}(\Omega^*,F^*)$.

Now recall that $\dq_w=\dq_{max}$ on $L^{n-p,n-q-1}(\Omega^*,F^*)$ and that
$\dq_{max}^* = \theta_{min}$ (see \eqref{eq:adjoint1}).
Thus, \eqref{eq:duality10},\eqref{eq:duality11} 
just means that $\chi\eta \in \Dom \dq_{min} \cap L^{p,q}(\Omega^*,F)$ with $\dq_{min}(\chi\eta)=\dq\chi\wedge \eta$.
But then we have on $K$ also $\eta \in \Dom\dq_s(K)$ because $\chi\equiv 1$ on $K$.
\end{proof}

\bigskip
\subsection{The closed range condition}

We are clearly interested in replacing the separated cohomology groups $\o{H}=\ker/\o{\Im}$ in Theorem \ref{thm:duality1}
and Theorem \ref{thm:duality2} by the 'real' cohomology $H=\ker/\Im$.
So, we need to study closed range conditions for the $\dq$-operators under consideration.
As a preparation, let us note:

\begin{lem}\label{lem:openmapping}
Let
\begin{eqnarray}\label{eq:rev13}
\dq_{w/s}: \Dom\dq_{w/s} \cap L^{p,q}_{loc}(\Omega,F) \longrightarrow L^{p,q+1}_{loc}(\Omega,F)
\end{eqnarray}
have closed range. 

Then the inverse mapping (i.e. the corresponding $\dq$-solution operator)
\begin{eqnarray*}
L=(\dq_{w/s})^{-1}:  \Im\dq_{w/s} \longrightarrow \frac{\Dom\dq_{w/s} \cap L^{p,q}_{w/s,loc}(\Omega,F)}{\ker \dq_{w/s}}.
\end{eqnarray*}
is continuous.
\end{lem}

\begin{proof}
As $\dq_{w/s}$ is a closed operator,
the graph $\Gamma_{w/s}$ of \eqref{eq:rev13} is a closed subspace of 
the Fr{\'e}chet space $L^{p,q}_{loc}(\Omega,F) \times L^{p,q+1}_{loc}(\Omega,F)$.
Hence, again a Fr{\'e}chet space with the induced topology.
By assumption, the range of \eqref{eq:rev13}, $\Im \dq_{w/s}$, is a closed subspace
of $L^{p,q+1}_{loc}(\Omega,F)$, so also a Fr{\'e}chet space with the induced topology.
Hence, by the open mapping theorem, it follows that the projection operator
\begin{eqnarray*}
\pi^{q+1}: \Gamma_{w/s} \longrightarrow \Im\dq_{w/s}
\end{eqnarray*}
is open. Thus, the induced mapping
\begin{eqnarray*}
\o{\pi^{q+1}}: \Gamma_{w/s}/\ker \pi^{q+1} \longrightarrow \Im\dq_{w/s}
\end{eqnarray*}
is again open and its inverse mapping
\begin{eqnarray*}
\big(\o{\pi^{q+1}}\big)^{-1}: \Im\dq_{w/s} \longrightarrow \Gamma_{w/s}/\ker \pi^{q+1}
\end{eqnarray*}
is continuous. Let moreover $\o{\pi^q}$ be the other induced projection
\begin{eqnarray*}
\o{\pi^q}: \frac{\Gamma_{w/s}}{\ker \pi^{q+1}} \longrightarrow
\frac{\Dom\dq_{w/s} \cap L^{p,q}_{w/s,loc}(\Omega,F)}{\ker \dq_{w/s}},
\end{eqnarray*}
which is also continuous. So, we obtain that $\dq_{w/s}^{-1}:=\o{\pi^q} \circ \big( \o{\pi^{q+1}}\big)^{-1}$
is actually a continuous operator.
\end{proof}

\medskip
\begin{lem}\label{lem:range}
If $\dq_{w/s}: \Dom\dq_{w/s}\cap L^{p,q}_{loc}(\Omega,F) \rightarrow L^{p,q+1}_{loc}(\Omega,F)$ has closed range,
then 
$$\dq_{s/w}:\ \Dom\dq_{s/w}\cap L^{n-p,n-q-1}_{cpt}(\Omega,F^*) \longrightarrow L^{n-p,n-q}_{cpt}(\Omega,F^*)$$
has closed range, too.
\end{lem}

\begin{proof}
Let $\omega\in\o{\Im\dq_{s/w}} \cap L^{n-p,n-q}_{cpt}(\Omega,F^*)$. We will show that $\omega\in\Im\dq_{s/w}$.

\smallskip
By Theorem \ref{thm:duality}, 
$\omega$ represents a continuous linear functional on $L^{p,q}_{loc}(\Omega,F)$,
and by partial integration, as in the proof of Theorem \ref{thm:duality1},
one sees that this functional vanishes on $\ker\dq_{w/s}$ because $\dq_{s/w}\omega=0$.

Thus, $\omega$ represents a continuous linear functional
\begin{eqnarray}\label{eq:range01}
\o{\omega}:\ \frac{L^{p,q}_{loc}(\Omega,F)}{\ker \dq_{w/s}} \rightarrow \C\ ,\ [\eta] \mapsto \int_{\Omega^*} \eta\wedge \omega.
\end{eqnarray}
By continuity of the mapping $L$ from Lemma \ref{lem:openmapping}, we obtain a continuous linear functional
\begin{eqnarray}\label{eq:range02}
\o{\omega}\circ L:\ \Im\dq_{w/s} \rightarrow \C\ ,\ f\mapsto \int_{\Omega^*} (\dq_{w/s})^{-1} f\wedge \omega.
\end{eqnarray}

By the Hahn-Banach theorem, $\o{\omega}\circ L$ extends to a continuous linear functional $\Lambda$ on
$L^{p,q+1}_{loc}(\Omega,F)$.
As such it is represented, again by Theorem \ref{thm:duality},
by a form $\lambda \in L^{n-p,n-q-1}_{cpt}(\Omega,F^*)$.

We claim that $\lambda \in\Dom\dq_{s/w}$ with $\dq_{s/w} \lambda = \omega$.
To see that, note that it follows from \eqref{eq:range01} and \eqref{eq:range02} that
\begin{eqnarray}\label{eq:new51}
\int_{\Omega^*} \eta\wedge \omega = \int_{\Omega^*} \dq_{w/s} \eta \wedge \lambda
\end{eqnarray}
for all $\eta\in \Dom\dq_{w/s} \cap L^{p,q}_{loc}(\Omega,F)$.

If $\dq_{s/w}$ denotes the operator $\dq_s$, then the claim follows now as in Theorem \ref{thm:duality1}, \eqref{eq:duality02},
because \eqref{eq:new51} implies that
\begin{eqnarray*}
(\dq_{max} g,\o{*}_{F^*}\omega)_{\Omega^*} &=&  \pm \int_{\Omega^*} \dq_w g\wedge \omega
= \pm \int_{\Omega^*} \dq_w^2 g\wedge \lambda=0
\end{eqnarray*}
for all $g\in \Dom\dq_{max} \cap L^{p,q-1}(\Omega^*,F)$.
Similarly, if $\dq_{s/w}$ denotes the operator $\dq_w$,
then the claim follows as in Theorem \ref{thm:duality2}, \eqref{eq:duality06}.
\end{proof}

\medskip
We will now prove that, conversely, the Hausdorff property of $H^{p,q+1}_{w/s,cpt}(\Omega,F^*)$
implies the Hausdorff property for $H^{n-p,n-q}_{s/w,loc}(\Omega,F)$.
This is more complicated because $L^{p,q}_{cpt}(\Omega,F^*)$ is just an (LF)-space,
and for such spaces it is not known whether closed subspaces or quotient spaces (modulo closed subspaces)
are again (LF)-spaces. So, we do not have a statement analogous to Lemma {\ref{lem:openmapping},
and we must find a way around this. Nevertheless, we can show:

\medskip
\begin{lem}\label{lem:range2}
If $\dq_{w/s}: \Dom\dq_{w/s}\cap L^{p,q}_{cpt}(\Omega,F^*) \rightarrow L^{p,q+1}_{cpt}(\Omega,F^*)$ has closed range,
then 
\begin{eqnarray}\label{eq:new100}
\dq_{s/w}:\ \Dom\dq_{s/w}\cap L^{n-p,n-q-1}_{loc}(\Omega,F) \longrightarrow L^{n-p,n-q}_{loc}(\Omega,F)
\end{eqnarray}
has closed range, too.
\end{lem}

\begin{proof}
Let $K_1 \subset\subset K_2\subset\subset K_3 ...$ be a compact exhaustion of $\Omega$.
So, the increasing sequences of Hilbert spaces
$\{L^{p,q}_{K_j}(\Omega,F^*)\}_j$ and $\{L^{p,q+1}_{K_j}(\Omega,F^*)\}_j$, respectively,
are defining sequences
for the (LF)-spaces $L^{p,q}_{cpt}(\Omega,F^*)$ and $L^{p,q+1}_{cpt}(\Omega,F^*)$, respectively.
Note that $L^{r,s}_{K_j}(\Omega,F^*) \cong L^{r,s}(K_j^*,F^*)$ by trivial extension of forms.

Let $\Gamma_{w/s}$ be the graph of \eqref{eq:new100}. As $\dq_{w/s}$ is a closed operator,
it is a closed subspace
of the (LF)-space $L^{p,q}_{cpt}(\Omega,F^*) \times L^{p,q+1}_{cpt}(\Omega,F^*)$.
Hence, we obtain an exhausting sequence of Fr{\'e}chet spaces 
\begin{eqnarray*}
E_\mu &:=& \Gamma_{w/s} \cap L^{p,q}_{K_\mu}(\Omega,F^*) \times L^{p,q+1}_{K_\mu}(\Omega,F^*)
\end{eqnarray*}
of $\Gamma_{w/s}$. Moreover, as $\Im \dq_{w/s}$ is closed,
\begin{eqnarray*}
F_\nu &:=& \Im \dq_{w/s} \cap L^{p,q+1}_{K_\nu}(\Omega,F^*)
\end{eqnarray*}
is a sequence of Fr{\'e}chet spaces exhausting $\Im\dq_{w/s}$.
Consider now the projection operator
\begin{eqnarray*}
\pi^{q+1}: && L^{p,q}_{cpt}(\Omega,F^*) \times L^{p,q+1}_{cpt}(\Omega,F^*) \longrightarrow  L^{p,q+1}_{cpt}(\Omega,F^*),
\end{eqnarray*}
and let
\begin{eqnarray*}
u:= \pi^{q+1}|_{\Gamma_{w/s}}: && \Gamma_{w/s} \longrightarrow \Im\dq_{w/s},
\end{eqnarray*}
which is clearly continuous and surjective. Now let
\begin{eqnarray}
G_{\mu\nu} &:=& E_\mu \cap u^{-1}(F_\nu),
\end{eqnarray}
which are again Fr{\'e}chet spaces with the induced topology because $u$ is continuous
(so that $u^{-1}(F_\nu)$ is closed).

We now follow the main argument of the proof of the open mapping theorem for (LF)-spaces
in \cite{DS}, Theorem 1. As $u(G_{\mu\nu})=u(E_\mu)\cap F_\nu$ and $u(\Gamma_{w/s})=\Im\dq_{w/s}$,
it follows that
\begin{eqnarray*}
F_\nu &=& \bigcup_\mu u(G_{\mu\nu}).
\end{eqnarray*}
But then at least one of the sets $u(G_{\mu\nu})$ must be of the second category (nonmeager)
in $F_\nu$ because $F_\nu$ is a Fr{\'e}chet space. For any $\nu$, choose as $\mu(\nu)$
the minimal $\mu$ such that this is the case.
Then it follows by the open mapping theorem (see e.g. \cite{Rd}, 2.11), that
\begin{eqnarray}\label{eq:open01}
u|_{G_{\mu(\nu)\nu}}: && G_{\mu(\nu)\nu} \longrightarrow F_\nu
\end{eqnarray}
is surjective and open.

Let us consider now the operator
\begin{eqnarray*}
\o{u}: && \Gamma_{w/s}/\ker\dq_{w/s} \longrightarrow \Im\dq_{w/s},
\end{eqnarray*}
which is again continuous, and define (as in Lemma \ref{lem:openmapping}) the $\dq$-solution operator
\begin{eqnarray*}
L=\o{\pi^{q}} \circ \o{u}^{-1}: && \Im\dq_{w/s} \longrightarrow \frac{\Gamma_{w/s}}{\ker\dq_{w/s}} 
\longrightarrow \frac{\Dom\dq_{w/s} \cap L^{p,q}_{cpt}(\Omega,F^*)}{\ker \dq_{w/s}},
\end{eqnarray*}
where the continuous operator $\o{\pi^q}$ is induced from the projection $\pi^q$ 
of $L^{p,q}_{cpt}(\Omega,F^*)\times L^{p,q}_{cpt}(\Omega,F^*)$ on the first factor.
It is now not clear whether $L$ is continuous because we cannot apply the open mapping theorem to $u$
as we do not know whether $\Im \dq_{w/s}$ is again an (LF)-space with the induced topology.
But, at least, by use of \eqref{eq:open01}
\begin{eqnarray}\label{eq:open02}
L_\nu:=L|_{F_\nu}:  \Im\dq_{w/s} \cap L^{p,q+1}_{K_\nu}(\Omega,F^*) = F_\nu 
\longrightarrow \frac{\Dom\dq_{w/s} \cap L^{p,q}_{K_{\mu(\nu)}}(\Omega,F^*)}{\ker \dq_{w/s}}
\end{eqnarray}
is a continuous linear Hilbert space operator for each $\nu$.
This is an interesting observation. \eqref{eq:open02} implies that given a compact set $K\subset\subset \Omega$,
there exists another compact set $K'\subset\subset \Omega$, such that for any $f\in \Im\dq_{w/s} \cap L^{p,q+1}_{K}(\Omega,F^*)$,
there exists $g\in L^{p,q}_{K'}(\Omega,F^*)$ with $\dq_{w/s} g=f$.

\medskip
Let now $\omega\in\o{\Im\dq_{s/w}} \subset L^{n-p,n-q}_{loc}(\Omega,F)$. We will show that $\omega\in\Im\dq_{s/w}$.
By Theorem \ref{thm:duality}, 
$\omega$ represents a continuous linear functional on $L^{p,q}_{cpt}(\Omega,F^*)$,
and by partial integration, as in the proof of Theorem \ref{thm:duality1},
one sees that this functional vanishes on $\ker\dq_{w/s} \cap L^{p,q}_{cpt}(\Omega,F^*)$ because $\dq_{s/w}\omega=0$.

Thus, $\omega$ represents a continuous linear functional
\begin{eqnarray}\label{eq:range01b}
\o{\omega}:\ \frac{L^{p,q}_{cpt}(\Omega,F^*)}{\ker \dq_{w/s}} \rightarrow \C\ ,\ [\eta] \mapsto \int_{\Omega^*} \eta\wedge \omega,
\end{eqnarray}
and we obtain a linear operator
\begin{eqnarray}\label{eq:range02b}
\Lambda:=\o{\omega} \circ L:\ \Im\dq_{w/s} \cap L^{p,q+1}_{cpt}(\Omega,F^*)
\rightarrow \C\ ,\ f\mapsto \int_{\Omega^*} L f\wedge \omega.
\end{eqnarray}
Again, we do not know whether $\Lambda$ is continuous, but by use of \eqref{eq:open02},
the restrictions
\begin{eqnarray}\label{eq:range03b}
\Lambda_\nu:=\Lambda|_{F_\nu}=\o{\omega} \circ L_\nu:\ \Im\dq_{w/s} \cap L^{p,q+1}_{K_\nu}(\Omega,F^*)
\rightarrow \C\ ,\ f\mapsto \int_{\Omega^*} L_\nu f\wedge \omega.
\end{eqnarray}
are continuous.

We wish to extend $\Lambda$ to a continuous linear functional on $L^{p,q+1}_{cpt}(\Omega,F)$.
This cannot be done in just one step because we do not know whether $\Lambda$ is continuous.
But we can use the Hahn-Banach theorem to extend $\Lambda_1$ to a continuous linear functional $\Lambda_1''$
on $L^{p,q+1}_{K_1}(\Omega,F)$. As $\Lambda_2|_{F_1}=\Lambda_1$,
we can extend $\Lambda_2$ by use of $\Lambda_1''$ to a continuous linear operator
$\Lambda_2'$ on the span of $L^{p,q+1}_{K_1}(\Omega,F)$ and $F_2$ in $L^{p,q+1}_{cpt}(\Omega,F)$.
Now use again Hahn-Banach to extend $\Lambda_2'$ to a continuous linear functional $\Lambda_2''$
on $L^{p,q+1}_{K_2}(\Omega,F)$. Going on inductively, we obtain a linear functional $\Lambda''$ on $L^{p,q+1}_{cpt}(\Omega,F)$
which is continuous on $L^{p,q+1}_{K_\nu}(\Omega,F)$ for all $\nu$.
So, $\Lambda''$ is by \cite{DS}, Proposition 5, a continuous linear extension of $\Lambda$.
This shows that $\Lambda$ is actually continuous.
By Theorem \ref{thm:duality}, $\Lambda''$ is represented by a form $\lambda \in L^{n-p,n-q-1}_{loc}(\Omega,F)$.

\medskip
We claim that $\lambda \in\Dom\dq_{s/w}$ with $\dq_{s/w} \lambda = \omega$.
To see that, note that it follows from \eqref{eq:range01b} and \eqref{eq:range02b} that
\begin{eqnarray}\label{eq:new41}
\int_{\Omega^*} \eta\wedge \omega = \int_{\Omega^*} \dq_{w/s} \eta \wedge \lambda
\end{eqnarray}
for all $\eta\in \Dom\dq_{w/s} \cap L^{p,q}_{cpt}(\Omega,F^*)$.
If $\dq_{s/w}$ denotes the operator $\dq_s$, then the claim follows now as in Theorem \ref{thm:duality1}, \eqref{eq:duality02}.
If $\dq_{s/w}$ denotes the operator $\dq_w$, then the claim follows as in Theorem \ref{thm:duality1}, \eqref{eq:duality04}.
\end{proof}

\medskip
\subsection{Proof of Theorem \ref{thm:main1}}

The proof of Theorem \ref{thm:main1} is now trivial.
Assume that $H^{p,q}_{w/s,loc}(\Omega,F)$ and $H^{p,q+1}_{w/s,loc}(\Omega,F)$ are Hausdorff.
Then $H^{n-p,n-q}_{s/w,cpt}(\Omega,F)$ is also Hausdorff by Lemma \ref{lem:range},
and so Theorem \ref{thm:duality1} and Theorem \ref{thm:duality2} give Theorem \ref{thm:main1}.

\medskip
\subsection{Topology of compact convergence on canonical sheaves}

Recall from Section \ref{ssec:complexes}
the canonical sheaves $\mathcal{K}_X = \ker \dq_w \subset \mathcal{C}^{n,0}$ and $\mathcal{K}_X^s=\ker\dq_s \subset \mathcal{F}^{n,0}$.
So, for any open set $\Omega \subset X$, 
$$\mathcal{K}_X(\Omega) = \ker\dq_w \cap L^{n,0}_{loc}(\Omega)$$
and
$$\mathcal{K}_X^s(\Omega) = \ker\dq_s \cap L^{n,0}_{loc}(\Omega)$$
carry the induced Fr\'echet space structure of $L^2$-convergence on compact subsets of $\Omega$.
It is clear that restriction maps are continuous, and so $\mathcal{K}_X$ and $\mathcal{K}_X^s$
are Fr\'echet sheaves according to Definition \ref{defn:fs}.

On the other hand, the Grauert-Riemenschneider canonical sheaf $\mathcal{K}_X$ is a coherent analytic sheaf.
So, $\mathcal{K}_X$ carries also the unique Fr\'echet sheaf structure of uniform convergence on compact subsets according
to Theorem \ref{thm:fs1} and Theorem \ref{thm:fs2}.
For an open set $\Omega\subset X$, let $\mathcal{K}_X(\Omega)$ carry the Fr\'echet space structure of $L^2$-convergence on compact subsets,
and $\widehat{\mathcal{K}_X}(\Omega)$ be the same algebraic vector space with the Fr\'echet space structure of uniform convergence on compact subsets.
Then the identity map
$$\widehat{\mathcal{K}_X}(\Omega) \longrightarrow \mathcal{K}_X(\Omega)$$
is bijective and continuous. So, the two Fr\'echet space structures coincide by the open mapping theorem.
To be more precise: It is enough to study the question locally (i.e. for small $\Omega$).
So, consider an epimorphism $\beta: \OO^N_X(\Omega) \rightarrow \mathcal{K}_X(\Omega)$, where $\OO^N_X$ carries the unique
Fr\'echet sheaf structure of uniform convergence on compact subsets. This map $\beta$ is clearly continuous.
So, by the open mapping theorem, the Fr\'echet space structures of $\mathcal{K}_X(\Omega)$ and $\OO^N_X(\Omega) / \ker\beta$
coincide. But $\OO^N_X(\Omega) / \ker\beta \cong \widehat{\mathcal{K}_X}(\Omega)$ topologically by Theorem \ref{thm:fs1}.

Concerning the canonical sheaf with boundary condition, $\mathcal{K}_X^s$,
the same argument holds as soon as we knew that the sheaf is coherent.
By now, this is just proven in the case of isolated singularities (see \cite{eins}, Theorem 1.10).
So, we summarize:

\begin{thm}\label{thm:topologies1}
Let $X$ be a Hermitian complex space. Then, on the Grauert-Riemen\-schneider canonical sheaf $\mathcal{K}_X$,
the Fr\'echet sheaf structure of $L^2$-convergence on compact subsets and the Fr\'echet sheaf structure of
uniform convergence on compact subsets coincide.

If $X$ has only isolated singularities, then the same holds for the canonical sheaf with Dirichlet boundary condition,
$\mathcal{K}_X^s$.
\end{thm}

\medskip
\subsection{Topological equivalence of $L^2$-cohomology and \v{C}ech cohomology}\label{ssec:topologies}

Let $X$ be a reduced complex space. Then any coherent analytic sheaf $\mathcal{S}\rightarrow X$ carries
a unique canonical Fr\'echet sheaf structure (with the topology of uniform compact convergence, see Section \ref{ssec:fs}).

By taking a Leray covering,
this induces a canonical Fr\'echet space topology on the \v{C}ech cohomology groups $\check{H}^q(\Omega,\mathcal{S})$
for $\Omega\subset X$ open.
We will now show that this Fr\'echet space topology coincides with our $L^2$-topology
when we have a suitable Dolbeault isomorphism. The central tool needed for that is the following statement:

\begin{lem}\label{lem:topologies}{\bf (\cite{RR}, Lemma 1)}
Let $A^*$ and $B^*$ be two complexes of Fr\'echet spaces
with continuous linear differentials, and let $u: A^* \rightarrow B^*$ be a continuous linear morphism
from $A^*$ to $B^*$. If $u$ is an algebraic quasi-isomorphism,
i.e., if it induces an algebraic isomorphism of the cohomology groups of the complexes,
then $u$ is also a topological quasi-isomorphism,
i.e. it induces a topological isomorphism of the cohomology groups carrying
their natural induced topology.
\end{lem}

\medskip
Let $X$ be a (possibly singular) Hermitian complex space (this includes particularly the case of a Hermitian complex manifold)
and consider the fine resolution
\begin{eqnarray*}\label{eq:complex1b}
0\rightarrow \mathcal{K}_X \longrightarrow \mathcal{C}^{n,0} \overset{\dq_w}{\longrightarrow}
\mathcal{C}^{n,1} \overset{\dq_w}{\longrightarrow} \mathcal{C}^{n,2} \overset{\dq_w}{\longrightarrow} ...
\end{eqnarray*}
according to Theorem \ref{thm:exactness1}. Let $\Omega\subset X$ be an open set.
Then the abstract DeRham-isomorphism 
\begin{eqnarray}\label{eq:iso111}
\check{H}^q(\Omega,\mathcal{K}_X) \cong H^q\big(\Gamma(\Omega,\mathcal{C}^{n,*})\big) 
=  H^{n,q}_{w,loc}(\Omega)
\end{eqnarray}
can be realized explicitly (cf. \cite{De1}, IV.6 for the following).
Let $\mathcal{U}=\{U_\alpha\}$ be a Leray covering for $\Omega$,
and let $\{\chi_\alpha\}$ be a smooth partition of unity subordinate to $\mathcal{U}$.
Given a \v{C}ech cocycle $c\in C^q(\mathcal{U},\mathcal{K}_X)$, we define a \v{C}ech coycle $f\in C^0(\mathcal{U},\mathcal{C}^{n,q})$
by
\begin{eqnarray*}
f_\alpha := \sum _{\nu_0, ..., \nu_{q-1}} \dq \chi_{\nu_0}\wedge ...\wedge \dq\chi_{\nu_{q-1}} \cdot c_{\nu_0 \cdots \nu_{q-1} \alpha} \ \ \ \mbox{ on }\ \ U_\alpha.
\end{eqnarray*}
In fact, $f$ is a cocycle and defines a $\dq$-closed global section
\begin{eqnarray*}\label{eq:construction}
\phi_c = \sum_{\nu_q} \chi_{\nu_q} f_{\nu_q} 
= \sum_{\nu_0, ..., \nu_q} \chi_{\nu_q}  \dq \chi_{\nu_0}\wedge ...\wedge \dq\chi_{\nu_{q-1}} \cdot c_{\nu_0 \cdots \nu_{q-1} \nu_q} 
\in \ker\dq_w \cap L^{n,q}_{loc} (\Omega).
\end{eqnarray*}
This mapping
\begin{eqnarray*}
\Psi: C^q(\mathcal{U},\mathcal{K}_X) \longrightarrow \ker\dq_w\cap L^{n,q}_{loc} (\Omega)\ ,\ c \mapsto \phi_c,
\end{eqnarray*}
is continuous as $C^q(\mathcal{U},\mathcal{K}_X)$ carries the Fr\'echet space structure of uniform convergence and
$L^{n,q}_{loc}(\Omega)$ the Fr\'echet space structure of $L^2$-convergence (the contribution of the partition of unity is harmless).

Taking into account also the topological isomorphism $\check{H}^q(\Omega,\mathcal{K}_X) \cong \check{H}^q(\mathcal{U},\mathcal{K}_X)$,
the algebraic isomorphism \eqref{eq:iso111} is then realized by the induced mapping
\begin{eqnarray*}
[\Psi]: \check{H}^q(\mathcal{U},\mathcal{K}_X) \overset{\cong}{\longrightarrow} H^{n,q}_{w,loc}(\Omega)\ ,\ [c]\mapsto [\phi_c].
\end{eqnarray*}
But now $[\Psi]$ is also a topological isomorphism by Lemma \ref{lem:topologies}.

Verbatim the same argument can be applied to $\mathcal{K}_X^s$ if $X$ has only isolated singularities (see Theorem \ref{thm:exactness2}).
Summing up, we get:

\begin{thm}\label{thm:topology}
Let $X$ be a possibly singular Hermitian complex space, $\Omega\subset X$ open. 
Then the Dolbeault isomorphism induces a topological isomorphism
\begin{eqnarray*}
\check{H}^q(\Omega,\mathcal{K}_X) \overset{\cong}{\longrightarrow} H^{n,q}_{w,loc}(\Omega),
\end{eqnarray*}
where the \v{C}ech cohomology carries the canonical Fr\'echet space structure of uniform convergence on compact subsets
and $H^{n,q}_{w,loc}(\Omega)$ the Fr\'echet space structure of $L^2$-convergence on compact subsets.

If $X$ has only isolated singularities, then the analogous statement holds for the Dolbeault isomorphism
\begin{eqnarray*}
\check{H}^q(\Omega,\mathcal{K}_X^s) \overset{\cong}{\longrightarrow} H^{n,q}_{s,loc}(\Omega).
\end{eqnarray*}
\end{thm}

\smallskip
It is just for ease of notation that we did not incorporate a Hermitian line bundle $F\rightarrow X$ in this section and the last one.
Note that all what has been said holds as well for forms with values in such a line bundle.

\medskip
\subsection{About the Hausdorff property}
\label{ssec:hausdorff}

In order to apply Theorem \ref{thm:main1}, we need to discuss the Hausdorff property for
the cohomology groups in the continuous algebraic isomorphisms \eqref{eq:new11}, \eqref{eq:new12},
\eqref{eq:new11b} and \eqref{eq:new12b} from Section \ref{ssec:resolution}.

Let us first discuss this for \eqref{eq:new11} and the trivial bundle $F$, i.e., for
\begin{eqnarray}\label{eq:new21}
H^{n,q}_{w,loc}(\Omega) = H^q\big( \Gamma(\Omega,\mathcal{C}^{n,*})\big) \overset{\cong}{\longrightarrow} 
H^q\big(\Gamma(\pi^{-1}(\Omega),\mathcal{C}^{n,*})\big) = H^{n,q}_{w,loc}\big(\pi^{-1}(\Omega)\big).
\end{eqnarray}
The algebraic isomorphism \eqref{eq:new11b} is treated completely analogous
(and the line bundle does not matter at all).
We have to consider the continuous linear morphism of complexes
\begin{eqnarray*}
\pi^*: \big(\mathcal{C}^{n,*}(\Omega),\dq_w\big) \longrightarrow \big(\mathcal{C}^{n,*}(\pi^{-1}(\Omega)),\dq_w\big)
\end{eqnarray*}
which induces the algebraic isomorphism \eqref{eq:new21}.
Hence, \eqref{eq:new21} is also a topological isomorphism by Lemma \ref{lem:topologies},
and we note:

\begin{lem}\label{lem:topology2}
The algebraic isomorphisms \eqref{eq:new11} and \eqref{eq:new11b}
from Section \ref{ssec:resolution} are also topological isomorphisms.
\end{lem}

\medskip
Let us now discuss also the algebraic isomorphisms \eqref{eq:new12} and \eqref{eq:new12b}.
Here, we have to deal with quotient spaces of (LF)-spaces and it is not clear whether
Lemma \ref{lem:topologies} is valid for (LF)-spaces (the quotient space of an (LF)-space by a closed
subspace may not be an (LF)-space with the induced topology).

But we can use the following simple observation:

\begin{prop}\label{prop:hausdorff}
Let $f: A \rightarrow B$ be an injective continuous mapping between
topological spaces $A$ and $B$, and assume that $B$ is Hausdorff.

Then $A$ is Hausdorff, too.
\end{prop}

\begin{proof}
Let $x,y\in A$ with $x\neq y$. Then $f(x)\neq f(y)$ by injectivity of $f$, and there exist open neighbourhoods
$U_x$ and $U_y$ of $x$ and $y$ in $B$ such that $U_x\cap U_y=\emptyset$ because $B$ is Hausdorff.
But then $\pi^{-1}(U_x)$ and $\pi^{-1}(U_y)$ are distinct open neighbourhoods of $x$ and $y$.
\end{proof}

From this, it is easy to deduce at least:

\begin{lem}\label{lem:topology3}
If the right hand side of \eqref{eq:new12} or \eqref{eq:new12b}, respectively, 
is a Hausdorff topological vector space, then so is the left hand side.
\end{lem}

\smallskip
\subsection{Examples of separated cohomology groups and proof of Theorem \ref{thm:main2}}

If $Z$ is a compact complex space, 
then the Cartan-Serre theorem shows that the cohomology of coherent analytic sheaves on $Z$ is finite-dimensional,
in particular Hausdorff.
More generally,
we get the Hausdorff property if $Z$ is holomorphically convex:
in that case $\check{H}^*(Z,\mathcal{S})$ is Hausdorff for any coherent analytic sheaf $\mathcal{S}$
by \cite{P}, Lemma II.1
(to conclude the statement from \cite{P},
recall that a holomorphically convex space $Z$ has a (proper) Remmert reduction $\pi: Z \rightarrow Y$ such that $Y$ is Stein).

\medskip
Let us now prove Theorem \ref{thm:main2}. Let $X$ be a Hermitian complex space, $\dim X=n$,
$\Omega\subset X$ holomorphically convex and $\pi: M\rightarrow X$ a resolution of singularities.
Then $\pi^{-1}(\Omega)$ is again holomorphically convex and it follows as explained above that
\begin{eqnarray*}\label{eq:new30}
\check{H}^q(\Omega,\mathcal{K}_X)\ ,\ \check{H}^q(\Omega,\mathcal{K}^s_X)\ ,\ \check{H}^{n-q}\big(\pi^{-1}(\Omega),\OO_M\big)
\end{eqnarray*}
are Hausdorff for all $0\leq q\leq n$ 
(for $\mathcal{K}_X^s$, assume that $X$ has only isolated singularities so that $\mathcal{K}_X^s$ is coherent).

So, it follows from Theorem \ref{thm:topology} (which holds as well for the cohomology with values in $\OO_M$
on the smooth manifold $M$) that the $L^2$-cohomology groups
\begin{eqnarray}\label{eq:new31}
H^{n,q}_{w,loc}(\Omega)\ ,\ H^{n,q}_{s,loc}(\Omega)\ , \ H^{0,n-q}_{w,loc}\big(\pi^{-1}(\Omega)\big)
\end{eqnarray}
are Hausdorff for all $0\leq q\leq n$.

Thus, by Lemma \ref{lem:range},
\begin{eqnarray}\label{eq:new32}
H^{n,q}_{w,cpt}\big(\pi^{-1}(\Omega)\big) = H^{n,q}_{s,cpt}\big(\pi^{-1}(\Omega)\big)
\end{eqnarray}
is Hausdorff for all $0\leq q\leq n$ ($\dq_w$ and $\dq_s$ coincide on $M$ which has no singular set).

But now Lemma \ref{lem:topology3} yields that
\begin{eqnarray}\label{eq:new33}
H^{n,q}_{w,cpt}(\Omega)\ ,\ H^{n,q}_{s,cpt}(\Omega)
\end{eqnarray}
are Hausdorff for all $0\leq q\leq n$ (here, for the statement about the $\dq_s$-cohomology,
we have to assume that $X$ has only isolated homogeneous singularities, see \eqref{eq:new12b}).

From \eqref{eq:new31} and \eqref{eq:new33} we see that actually -- as claimed --
all the $H^{n,q}$-cohomology groups in the statement of Theorem \ref{thm:main2} are Hausdorff
(under the conditions imposed on the singularities of $X$ for the $\dq_s$-cohomology).

The dual $H^{0,n-q}$-cohomology groups are then Hausdorff by Lemma \ref{lem:range} and Lemma \ref{lem:range2}.
That completes the proof of Theorem \ref{thm:main2}.

\smallskip
There is another interesting example of Hausdorff cohomology.
Let $\Omega \subset X$ be $q$-convex. Then it follows by the Andreotti-Grauert theory that
$H^r(\Omega,\mathcal{S})$ is of finite dimension (hence Hausdorff)
for any coherent analytic sheaf $\mathcal{S}$ and all $r\geq q$.

\medskip
\section{Vanishing theorems for the $\dq$-cohomology on singular spaces}\label{sec:dqs}

\subsection{Proof of Theorem \ref{thm:vanishing1}}

Both statements follow simply from the combination of Theorem \ref{thm:topology}, Lemma \ref{lem:range}
and the two duality statements Theorem \ref{thm:duality1} and \ref{thm:duality2}.

\medskip
\subsection{On the domain of the $\dq_s$-operator}\label{ssec:domaindqs}

In view of our vanishing result, Theorem \ref{thm:vanishing1}, it would be good to understand the $\dq_s$-operator better,
but, in general, it is difficult to decide if a differential form is in the domain of the $\dq_s$-operator.
However, we have at least the following criterion:

\begin{thm}\label{thm:domaindqs2}
Locally bounded forms in the domain of the $\dq_w$-operator are also in the domain of the $\dq_s$-operator.
\end{thm}

Here, locally bounded means bounded on $K^*=K-\Sing X$ for compact sets $K\subset\subset X$.
A proof of Theorem \ref{thm:domaindqs2}, using resolution of singularities,
was given in \cite{PAMS}, Theorem 1.6.
But Theorem \ref{thm:domaindqs2} can be achieved also in a more direct way
as an easy consequence of Sibony's results on extension of analytic objects \cite{Sibony}.
The following short argument was communicated to the author by Nessim Sibony.

\begin{proof}
As the problem is local, consider a Hermitian complex space $X$ of dimension $n$,
embedded locally in some open set $\Omega\subset \C^N$.
Let $T$ be the positive closed current of integration over $X$, and $A$ the singular set of $X$ (as a subset of $\Omega$),
which is a complete pluripolar set.
By \cite{Sibony}, Lemma 1.2, there exists a sequence $\{u_j\}$ of smooth plurisubharmonic functions, $0\leq u_j\leq 1$,
vanishing identically on a neighborhood of $A$, converging uniformly to $1_{\Omega\setminus A}$
on compact subsets of $\Omega\setminus A$.

Consider now compact sets $K, L \subset \Omega$ with $K\subset \mbox{int}(L)$,
and let $0\leq \chi\leq 1$ be a test function
with support in $L$ and $\chi =  1$ on $K$. Let $\omega:=i\partial\dq \|z\|^2$ be the K\"ahler form in $\C^N$. 
Then, similarly as in the proof of Theorem 1.3 in \cite{DiSi}, p. 361, we have
\begin{eqnarray*}
\| \dq u_j\|^2_{L^2(K\cap X)} &\leq& \int \chi^2 i\partial u_j\wedge \dq u_j \wedge T \wedge \omega^{n-1}
\leq \frac{1}{2} \int \chi^2 i\partial \dq u_j^2 \wedge T \wedge \omega^{n-1}\\
&=& \frac{1}{2} \int \chi^2 i\partial \dq (u_j^2-1) \wedge T \wedge \omega^{n-1}
= \frac{1}{2} \int (u_j^2 -1) i\partial\dq \chi^2 \wedge T \wedge \omega^{n-1}\\
&\leq& C_\chi \|u_j^2-1\|_{L^1(L\cap X)} \overset{j\rightarrow \infty}{\longrightarrow} 0,
\end{eqnarray*}
where we have convergence to $0$ because the $u_j^2$ converge uniformly to $1$ on compact subsets of $\Omega\setminus A$.

So, let $f$ be a bounded form in the domain of $\dq_w$.
Then it is easy to see (using the H\"older inequality)
that $f_j:= u_j f$ is a sequence as required in \eqref{eq:ds1}, \eqref{eq:ds2}.
\end{proof}

\medskip
\subsection{Hartogs' extension theorem}\label{ssec:hartogs}

Let $X$ be a connected normal complex space of dimension $n\geq 2$ which is cohomologically $(n-1)$-complete.
Then
\begin{eqnarray}\label{eq:hartogs1}
H^{0,1}_{s,cpt}(X) &=&  0
\end{eqnarray}
by Theorem \ref{thm:vanishing1}. As in \cite{Hoe2}, Section 2.3,
one can use \eqref{eq:hartogs1} to give a short proof of Hartogs'
extension theorem in its most general form by the $\dq$-method of Ehrenpreis.
We repeat the argument for convenience of the reader.

Let $D$ be a domain in $X$ and $K\subset D$ a compact subset such that $D\setminus K$ is connected.
Let $f\in \OO(D\setminus K)$. We claim that $f$ has a unique holomorphic extension to the whole domain $D$.

To show that, let $\chi \in C^\infty_{cpt}(X)$ be a smooth cut-off function that is identically $1$ in a neighborhood of $K$
such that $C:=\supp \chi \subset\subset D$.

Consider
$$g:=(1-\chi) f \in C^\infty(D),$$
which is an extension of $f$ to $D$, but unfortunately not holomorphic. We have to fix it by
the idea of Ehrenpreis. So, let
$$\omega:=\dq g,$$
which is a bounded $\dq_w$-closed $(0,1)$-form with compact support in $D$.
By Theorem \ref{thm:domaindqs2}, $\omega$ is also $\dq_s$-closed.
We may consider $\omega$ as an $L^2$-form on $X$ with compact support.

But $H^{0,1}_{s,cpt}(X)=0$ as seen above, \eqref{eq:hartogs1}.
So, there exists $h\in L^{0,0}_{cpt}(X^*)$ such that
$$\dq_s h=\omega,$$
and $h$ is holomorphic on $X^* \setminus C$ (where $\omega\equiv 0$).
Since $X$, being $(n-1)$-complete, is non-compact,
it follows by standard arguments (involving the identity theorem) that $h\equiv 0$ on an open subset of $D\setminus C\subset D\setminus K$.
Note that $X$ is normal, thus locally irreducible, and so connected subsets of $X$ satisfy the identity theorem.

Let
\begin{eqnarray*}
F:=(1-\chi)f-h \in \OO(X^*).
\end{eqnarray*}
As $X$ is normal, $F$ extends by the Riemann extension theorem to a holomorphic function on $X$, say $F\in\OO(X)$ for ease of notation
(see e.g. \cite{CAS}, Chapter 7.4.2).
As $h$ is vanishing on an open subset of $D\setminus K$,
$F$ is the desired unique extension by the identity theorem.

\smallskip
\subsection{Solution of the $\dq$-equation for bounded forms with compact support}

By the method applied in the last section, we obtain also:

\begin{thm}\label{thm:vanishing2}
Let $X$ be a Hermitian complex space of pure dimension $n$, $F\rightarrow X$ a Hermitian holomorphic line bundle,
$\Omega\subset X$ a cohomologically $q$-complete subset.
Let $1\leq r \leq n-q$ and $f$ an $F$-valued bounded $(0,r)$-form on $\Omega^*$ with compact support in $\Omega$ that is $\dq$-closed in the sense of distributions.
Then there exists an $F$-valued $L^2$-form with compact support in $\Omega$, $h\in L^{0,r-1}_{cpt}(\Omega,F)$, s.t. $\dq h=f$
in the sense of distributions on $\Omega^*$.
\end{thm}

\begin{proof}
Just combine Theorem \ref{thm:vanishing1} and Theorem \ref{thm:domaindqs2}.
\end{proof}

For $(0,1)$-forms, we can say more, using a resolution of singularities:

\begin{thm}\label{thm:vanishing3}
In the situation of Theorem \ref{thm:vanishing2}, let $\Omega$ be cohomologically $(n-1)$-complete and $f$ a bounded $(0,1)$-form.
Then the solution $h$ can be chosen to be a bounded function on $\Omega^*$. We write for that:
$$H^{0,1}_{L^\infty,cpt}(\Omega^*,F)=0.$$
\end{thm}

\begin{proof}
Consider a resolution of singularities $\pi: \Omega'\rightarrow \Omega$.
Then $\pi^* f$ is a bounded $(0,1)$-form and $\dq$-closed in the sense of distributions by
\cite{Rp1}, Theorem 3.2 (the $\dq$-equation extends over the exceptional set).
But $H^{0,1}_{cpt}(\Omega',\pi^* F)=0$ by \cite{CR}, Theorem 2.6.
So, there exists a bounded function $h$ with compact support and $\dq h=\pi^* f$.
Pushing forward $h$ outside the exceptional set gives the desired bounded solution.
\end{proof}

\section{$L^2$-cohomology and rational singularities}\label{sec:rational}

\subsection{Proof of Theorem \ref{thm:main3a}}\label{ssec:main3a}

Irreducible components of different dimension can be treated separately,
so we can assume that $X$ is of pure dimension $n$.

Let $\Omega\subset X$ be holomorphically convex.
By Theorem \ref{thm:main2} (and its proof) the cohomology groups
$H^{0,q}_{s,loc}(\Omega)$, $H^{n,n-q}_{w,cpt}(\Omega)$, $H^{0,q}_{w,loc}\big(\pi^{-1}(\Omega)\big)$,
$H^{n,n-q}_{w,cpt}\big(\pi^{-1}(\Omega)\big)$
are Hausdorff for all $0\leq q \leq n$.

So, there exist by Theorem \ref{thm:main1} non-degenerate pairings
(recall that $\dq_w$ and $\dq_s$ coincide on $M$ as there are no singularities at all)
\begin{eqnarray}\label{eq:niso1}
H^{0,q}_{s,loc}(\Omega) \times H^{n,n-q}_{w,cpt}(\Omega) \longrightarrow \C
\end{eqnarray}
and
\begin{eqnarray}\label{eq:niso2}
H^{0,q}_{w,loc}\big(\pi^{-1}(\Omega)\big) \times H^{n,n-q}_{w,cpt}\big(\pi^{-1}(\Omega)\big) \longrightarrow \C.
\end{eqnarray}
But $\pi$ induces by pull-back of $(n,q)$-forms natural isomorphisms
$$[\pi^*]: H^{n,n-q}_{w,cpt}(\Omega) \overset{\cong}{\longrightarrow} H^{n,n-q}_{w,cpt}\big(\pi^{-1}(\Omega)\big)$$
(see \eqref{eq:new12}).
So, \eqref{eq:niso1} and \eqref{eq:niso2} induce dual isomorphisms
\begin{eqnarray}\label{eq:new61}
H^q\big(\pi^{-1}(\Omega),\OO_M\big) \cong H^{0,q}_{w,loc}\big(\pi^{-1}(\Omega)\big) \overset{\cong}{\longrightarrow}
H^{0,q}_{s,loc}(\Omega)
\end{eqnarray}
for all $0\leq q\leq n$. \eqref{eq:new61} is induced by the push-forward of $L^2$-forms under $\pi^{-1}$
(defined outside the exceptional set), and it is continuous by \cite{eins}, (14).
Thus, \eqref{eq:new61} is a topological isomorphism by the open mapping theorem.

\subsection{Proofs of Corollary \ref{cor:main1} and of Corollary \ref{cor:main2}}

Both statements are simple corollaries of Theorem \ref{thm:main3a}. For the proof of Corollary \ref{cor:main1},
just recall that
$$\big(\mathcal{H}^0(\mathcal{F}^{0,*})\big)_x = \lim_{\substack{\longrightarrow\\ x\in\Omega}} H^{0,q}_{s,loc}(\Omega)$$
and
$$\big( R^q \pi_* \OO_M\big)_x = \lim_{\substack{\longrightarrow\\ x\in\Omega}} H^q\big(\pi^{-1}(\Omega),\OO_M\big).$$
Recall also that $(\pi_* \OO_M)_x = \widehat{\OO}_{X,x}$,
where $\widehat{\OO}_X$ denotes the sheaf of weakly holomorphic functions.

Corollary \ref{cor:main2} is now in turn a simple corollary of Corollary \ref{cor:main1},
keeping in mind that a point $x\in X$ is by definition rational if it is normal and $\big( R^q \pi_* \OO_M\big)_x =0$ for all $q>0$.

\smallskip
\subsection{Homogeneous isolated singularities}

For the sake of completeness, let us include also the following observation:

\begin{thm}\label{thm:isolated}
Let $X$ be a Hermitian complex space with only homogeneous (conical) isolated singularities,
$\pi: M\rightarrow X$ a resolution of singularities and $\Omega \subset X$ holomorphically convex.
Then push-forward of forms induces for all $q\geq 0$ a natural topological isomorphism
$H^q\big(\pi^{-1}(\Omega),\OO_M\big) \overset{\cong}{\longrightarrow} H^{0,q}_{w,loc}(\Omega)$.
The $L^2$-$\dq$-complex
\begin{eqnarray}\label{eq:exactnessM002}
0\rightarrow \OO_X \longrightarrow \mathcal{C}^{0,0} \overset{\dq_w}{\longrightarrow}
\mathcal{C}^{0,1} \overset{\dq_w}{\longrightarrow} \mathcal{C}^{0,2} 
\overset{\dq_w}{\longrightarrow} \mathcal{C}^{0,3} \overset{\dq_w}{\longrightarrow} ...
\end{eqnarray}
is exact in a point $x\in X$ if and only if $x$ is a rational point.
\end{thm}

\begin{proof}
If $X$ has only homogeneous isolated singularities, then $\mathcal{K}_X \cong \mathcal{K}_X^s$
by \cite{eins}, Theorem 1.10, with $D=\emptyset$ because homogeneous isolated singularities can be resolved by a single blow-up.
So, the statements can be seen completely analogous to the proofs of Theorem \ref{thm:main3a}
and of Corollary \ref{cor:main2}.

Note that clearly $\OO_{X,x} \subset \widehat{\OO}_{X,x} \subset \big(\ker \dq_w^{0,0}\big)_x$. 
Moreover, $\big(\ker\dq_w^{0,0}\big)_x \subset \OO_{X,x}$ 
if $X$ is normal by the Riemann Extension Theorem for singular spaces 
(by normality, the codimension of the singular set is bigger than one).
So, a point $x\in X$ is normal exactly if $\big(\ker\dq_w^{0,0}\big)_x = \OO_{X,x}$.
\end{proof}

{\bf Acknowledgments.}
This research was supported by the Deutsche Forschungsgemeinschaft (DFG, German Research Foundation), 
grant RU 1474/2 within DFG's Emmy Noether Programme.
The author thanks Nessim Simony for pointing out the short proof of Theorem \ref{thm:domaindqs2},
and he is grateful to an unknown referee for many valuable comments on an earlier version of this paper,
which helped to improve the presentation considerably.




\begin{thebibliography}{99999}

\scriptsize



\bibitem[AZ1]{AZ1} {\sc F.\ Acosta, E.\ Zeron},
H\"older estimates for the $\dq$-equation on surfaces with simple singularities,
{\em Bol. Soc. Mat. Mexicana} {\bf 12} (2006), no.2, 193--204.

\bibitem[AZ2]{AZ2} {\sc F.\ Acosta, E.\ Zeron},
H\"older estimates for the $\dq$-equation on surfaces with singularities of the type E6 and E7,
{\em Bol. Soc. Mat. Mexicana} {\bf 13} (2007), no.1, 73--86.

 
\bibitem[A]{Alt} {\sc H.\ W.\ Alt}, {\em Lineare Funktionalanalysis}, Springer-Verlag, Berlin, 1992.

\bibitem[AS]{AS}{\sc M.\ Andersson, H.\ Samuelsson},
A Dolbeault-Grothendieck lemma on complex spaces via Koppelman formulas,
{\em Invent. Math.} {\bf 190} (2012), no. 2, 261--297.


\bibitem[AG]{AG}{\sc A.\ Andreotti, H.\ Grauert}, Th\'eor\`eme de finitude pour la cohomologie des espaces complexes,
{\em Bull. Soc. Math. France} {\bf 90} (1962), 193--259.


\bibitem[AK]{AK}{\sc A.\ Andreotti, A.\ Kas},
Duality on complex spaces, {\em Ann. Scuola Norm. Sup. Pisa (3)} {\bf 27} (1973), 187--263.




\bibitem[AV]{AnVe} {\sc A.\ Andreotti, E.\ Vesentini},
Carleman estimates for the Laplace Beltrami equation on complex manifolds,
{\em Publ. Math. Inst. Hautes Etudes Sci.} {\bf 25} (1965), 81--130.





\bibitem[CR]{CR} {\sc M.\ Coltoiu, J.\ Ruppenthal}, On Hartogs' extension theorem on $(n-1)$-complete spaces,
{\em J. reine angew. Math} {\bf 637} (2009), 41--47.




\bibitem[D]{De1}{\sc J.-P.\ Demailly}, {\em Complex Analytic and Differential Geometry},
online book, available at {\sf www-fourier.ujf-grenoble.fr/$\sim$demailly/manuscripts/agbook.pdf}, Institut Fourier, Grenoble.




\bibitem[DS1]{DS} {\sc J.\ Dieudonn{\'e}, L.\ Schwartz},
La dualit\'e dans les espaces (F) et (LF), {\em Ann. Inst. Fourier Grenoble} {\bf 1} (1949), 61--101.


\bibitem[DS2]{DiSi}{\sc T.-C.\ Dinh, N.\ Sibony},
Pull-back of currents by holomorphic maps, {\em manuscripta math.} {\bf 123} (2007), 357--371.






\bibitem[FOV]{FOV2} {\sc J.\ E.\ Forn{\ae}ss, N.\ {\O}vrelid, S.\ Vassiliadou},
Local $L^2$ results for $\dq$: the isolated singularities case,
{\em Internat. J. Math.} {\bf 16} (2005), {\em no. 4}, 387--418.



\bibitem[GR1]{CAS} {\sc H.\ Grauert, R.\ Remmert}, {\em Coherent Analytic Sheaves},
Grundlehren der Mathematischen Wissenschaften, no. 265, Springer, 1984.


\bibitem[GR2]{GrRie} {\sc H.\ Grauert, O.\ Riemenschneider},
Verschwindungss\"atze f\"ur analytische Kohomologiegruppen auf komplexen R\"aumen,
{\em Invent. Math.} {\bf 11} (1970), 263--292.



\bibitem[GKKP]{GKKP} {\sc D.\ Greb, S.\ Kebekus, S.\ Kov{\'a}cs, T.\ Peternell},
Differential forms on log canonical spaces,
{\em Publ. Math. Inst. Hautes {\'E}tudes Sci.} {\bf 114} (2011), 87--169.


\bibitem[GR3]{GR} {\sc R.\ C.\ Gunning, H.\ Rossi},
{\em Analytic functions of several complex variables},
AMS Chelsea Publishing, Providence, RI, 2009.



\bibitem[H1]{Hoe1}{\sc L.\ H\"ormander}, $L^2$-estimates and existence theorems for the $\dq$-operator,
{\em Acta Math.} {\bf 113} (1965), 89--152.

\bibitem[H2]{Hoe2}{\sc L.\ H\"ormander}, {\em An introduction to complex analysis in several variables},
North-Holland Mathematical Library, 7. North-Holland Publishing Co., Amsterdam, 1990. 


\bibitem[K]{Kol}{\sc J.\ Koll\'ar}, Singularities of pairs. {\em Algebraic geometry -- Santa Cruz 1995}, 221--287,
Proc. Sympos. Pure Math. {\bf 62}, Part 1, Amer. Math. Soc., Providence, RI, 1997.


\bibitem[LR1]{LR}{\sc R.\ L\"ark\"ang, J.\ Ruppenthal}, Koppelman formulas on the $A_1$-singularity,
{\em J. Math. Anal. Appl.} {\bf 437} (2016), 214--240.


\bibitem[LR2]{LR2}{\sc R.\ L\"ark\"ang, J.\ Ruppenthal},
Koppelman formulas on affine cones over smooth projective complete intersections,
Preprint 2015, {\sf arXiv:1509.00987}, submitted.



\bibitem[MV]{MV}{\sc R.\ Meise, D.\ Vogt}, {\em Introduction to Functional Analysis},
Oxford Graduate Texts in Mathematics, Oxford University Press, New York, 1997. 


\bibitem[MP]{MePo2}{\sc J.\ Merker, E.\ Porten}, The Hartogs' extension theorem on $(n-1)$-complete complex spaces,
{\em J. reine u. angw. Math.} {\bf 637} (2009), 23--39.





\bibitem[OT]{OT} {\sc T.\ Ohsawa, K.\ Takegoshi},
On the extension of $L^2$-holomorphic functions,
{\em Math. Z.} {\bf 195} (1987), no. 2, 197--204.




\bibitem[OV1]{OV1} {\sc N.\ {\O}vrelid, S.\ Vassiliadou}, Solving $\dq$ on product singularities,
{\em Complex Var. Ellipitic Equ.} {\bf 51} (2006), {\em no. 3}, 225--237.


\bibitem[OV2]{OV4}{\sc N.\ {\O}vrelid, S.\ Vassiliadou},
Semiglobal results for $\dq$ on complex spaces with arbitrary singularities, Part II,
{\em Trans. Amer. Math. Soc.} {\bf 363} (2011), 6177--6196.


\bibitem[OV3]{OV3}{\sc N.\ {\O}vrelid, S.\ Vassiliadou}, $L^2$-$\dq$-cohomology groups of some singular complex spaces,
{\em Invent. Math.} {\bf 192} (2013), no.2, 413--458.



\bibitem[PS]{PS1}{\sc W.\ Pardon, M.\ Stern},
$L^2$-$\dq$-cohomology of complex projective varieties, {\em J. Amer. Math. Soc.} {\bf 4} (1991), no. 3, 603--621.



\bibitem[P]{P} {\sc D.\ Prill}, The divisor class groups of some rings of holomorphic functions,
{\em Math. Z.} {\bf 121} (1971), 58--80.


\bibitem[RR]{RR}{\sc J.-P.\ Ramis, G.\ Ruget}, Complexe dualisant et th\'eor\`eme de dualit\'e en g\'eom\'etrie analytique complexe,
{\em Inst. Hautes \'Etudes Sci. Publ. Math.} {\bf 38} (1970), 77--91.


\bibitem[R1]{Rd}{\sc W.\ Rudin}, {\em Functional Analysis}, 
International Series in Pure and Applied Mathematics, McGraw-Hill, New York, 1991.



\bibitem[R2]{Rp1}{\sc J.\ Ruppenthal}, About the $\dq$-equation at isolated singularities with regular exceptional set,
{\em Internat. J. Math.} {\bf 20} (2009), no. 4, 459--489.


\bibitem[R3]{Rp7}{\sc J.\ Ruppenthal}, The $\dq$-equation on homogeneous varieties with an isolated singularity,
{\em Math. Z.} {\bf 263} (2009), 447--472.


\bibitem[R4]{eins} {\sc J.\ Ruppenthal}, $L^2$-theory for the $\dq$-operator on compact complex spaces,
{\em Duke Math. J.} {\bf 163} (2014), 2887--2934.



\bibitem[R5]{PAMS} {\sc J.\ Ruppenthal},
Parabolicity of the regular locus of complex varieties,
{\em Proc. Amer. Math. Soc.} {\bf 144} (2016), 225--233.






\bibitem[RS]{RS} {\sc J.\ Ruppenthal, M.\ Sera},
$L^2$-Riemann-Roch for singular complex curves, 
{\em J. Singul.} {\bf 11} (2015), 67--84.







\bibitem[S1]{Serre}{\sc J.-P.\ Serre},
Un th{\'e}or{\`e}me de dualit{\'e}, {\em Comm. Math. Helv.} {\bf 29} (1955), 9--26.


\bibitem[S2]{Sibony}{\sc N.\ Sibony},
Quelques probl\`emes de prolongement de courants en analyse complexe,
{\em Duke Math. J.} {\bf 52} (1985), no. 1, 157--197.




\bibitem[S3]{Siu0}{\sc Y.-T.\ Siu},
Analyticity of sets associated to Lelong numbers and the extension of closed positive currents,
{\em Invent. Math.} {\bf 27} (1974), 53--156.

\bibitem[S4]{Siu}{\sc Y.-T.\ Siu},
Invariance of Plurigenera, {\em Invent. Math.} {\bf 134} (1998), no. 3, 661--673.


\bibitem[SZ]{SZ} {\sc M.\ Solis, E.\ Zeron},
H\"older estimates for the $\dq$-equation on singular quotient varieties,
{\em Bol. Soc. Mat. Mexicana} {\bf 14} (2008), no.1, 35--47.


\bibitem[T1]{Ta} {\sc K.\ Takegoshi}, Relative vanishing theorems in analytic spaces,
{\em Duke Math. J.} {\bf 51} (1985), no. 1, 273--279.


\bibitem[T2]{Tr} {\sc F.\ Tr\`eves},
{\em Topological vector spaces, distributions and kernels}.
Academic Press, New York-London, 1967.


\end{thebibliography}
\end{document}